\PassOptionsToPackage{dvipsnames}{xcolor}

\documentclass[11pt,authoryear,elsevier,numafflabel]{elsarticle}

\usepackage{ulem}
\usepackage{xcolor}
\usepackage{setspace}
\usepackage{pdflscape}
\usepackage[round]{natbib}
\usepackage{amsmath}
\usepackage{mathtools}
\usepackage{mathrsfs}
\usepackage{amssymb}
\usepackage{bbold}
\usepackage{booktabs}
\usepackage{multirow}
\usepackage{multicol}

\DeclareMathOperator*{\argmin}{argmin}

\usepackage{amsthm}
\newtheorem{theorem}{Theorem}[section]

\newtheorem*{remark}{Remark}
\newtheorem{corollary}{Corollary}[theorem]

\usepackage{algpseudocode}
\usepackage{algorithm}

\usepackage[english]{babel}
\usepackage{csquotes}
\usepackage[T1]{fontenc}
\usepackage[colorlinks=true, allcolors=blue]{hyperref}
\usepackage{subcaption}
\usepackage{fancyhdr}
\usepackage[toc,page]{appendix}

\usepackage{parskip}
\usepackage{a4wide}

\usepackage{tikz}
\usetikzlibrary[arrows.meta]
\usetikzlibrary {graphs,positioning,shapes.geometric,quotes}
\usetikzlibrary {datavisualization.formats.functions} 
\usetikzlibrary{decorations.pathreplacing}

\title{Solution of a bilevel optimistic scheduling problem on parallel machines}
\date{}

\author[aff1,aff2]{Quentin Schau\corref{cor1}}
\author[aff1]{Olivier Ploton}
\author[aff1]{Vincent T'kindt}
\author[aff4]{Han Hoogeveen}
\author[aff2,aff3]{Federico Della Croce}
\author[aff4]{Jippe Hoogeveen}

\address[aff1]{University of Tours, LIFAT, France.}
\address[aff2]{Politecnico di Torino, DIGEP, Torino, Italy.}
\address[aff3]{CNR, IEIIT, Torino, Italy}
\address[aff4]{Utrecht University, Department of Information and Computing Sciences, Netherlands.}
\cortext[cor1]{Corresponding author. \\ E-mail addresses: quentin.schau@univ-tours.fr; quentin.schau@polito.it (Q. Schau)}
\begin{document}

\begin{abstract}
	We consider the uniform parallel machines scheduling problem in the context of optimistic bilevel optimization, where two speed options are considered. In this scenario, the leader aims to minimize the weighted number of tardy jobs, while the follower seeks to minimize the total completion time on a set of uniform machines. This problem has practical applications in Industry 4.0. We show that this problem is $\cal NP$-hard in the strong sense by providing a reduction from the Numerical 3-Dimensional Matching problem and we provide a moderately exponential-time dynamic programming algorithm. The problem is solved by means of a concise MIP formulation and a branch-and-bound algorithm that embeds a column generation approach for the lower bound computation. Computational experiments are presented for instances with up to $80$ jobs and $4$ machines while larger problems are out of reach for the proposed approaches.
\end{abstract}


% Change bullet itemize puce
\def\labelenumi{({\roman{enumi}})}

\maketitle

\section{Introduction}

Bilevel scheduling problems model the situation in which two agents interact to schedule a set of jobs, but under a specific hierarchical setting. The first agent, called the leader, take part of the scheduling decisions to minimize its objective function $f^L$. Next, the second agent, called the follower, takes the remaining decisions to minimize its objective function $f^F$. The particularity of such problems is that these two objective functions can only be evaluated when all scheduling decisions have been taken. Bilevel scheduling is a challenging and emerging research topic that relates to bilevel optimization. The literature on such problems is rich and has focused on both continuous linear and non-linear problems \citep{bardPracticalBilevelOptimization1998, dempeFoundationsBilevelProgramming2002,dempeBilevelOptimizationTheory2020}. It appears that, when the leader has taken his decisions, the follower's problem may have multiple optimal solutions: this leads to different classes of bilevel problems. In optimistic ($OPT$) bilevel problems, the follower returns the solution that leads to the smallest value of the leader's objective function among all optimal solutions. In pessimistic ($PES$) bilevel problems, the follower returns the optimal solution that is the worst for the leader's objective function. Finally, in adversarial ($ADV$) bilevel problems, the leader takes decisions to ensure that the optimal solution of the follower's problem is the worst possible. 

To the best of our knowledge, the literature on bilevel scheduling is relatively limited. To solve a bilevel flow shop problem, \cite{karlofBilevelProgrammingApplied1996} proposed an exact two level branch-and-bound algorithm, where the leader assigns the operators to the machines in order to minimize the total flowtime, while the follower decides on the jobs' schedule in order to minimize the makespan. The considered approach was limited to instances with up to $10$ jobs.
\cite{abassBilevelProgrammingApproach2005a} studied an extension of this problem by considering fuzzy processing times.
\cite{kasemsetBilevelMultiobjectiveMathematical2010a} proposed a bilevel scheduling multi-objective job shop problem that operates under the Theory of Constraints (TOC) policy. It was extended in \cite{kasemsetPSObasedProcedureBilevel2012} who proposed a heuristic approach based on Particle Swarm Optimization. In this problem, the leader aims to minimize idle time on bottleneck machines before the follower, that schedules jobs in order to minimize the weighted aggregating function, i.e., the makespan, the maximum tardiness and the maximum earliness.
\cite{brownComplexityDelayingAdversarys2005} considered a bilevel project scheduling problem where the leader takes decisions in order to maximize the completion time of the follower. We can see this problem as an adversarial setting. They investigated the computational complexity of variants of the problem.
\cite{lukacProductionPlanningProblem2008} proposed a heuristic based on Tabu search to solve a bilevel scheduling problem, where the objective of the leader is to assign the products to the machines to minimize the total sequence dependent setup time, while the objective of the follower is to minimize the production, storage and setup costs of the machine.
\cite{kovacsConstraintProgrammingApproach2011a} used an exact constraint programming approach to model and solve discrete bilevel problems, and they applied it on optimistic bilevel single machine problems in which the leader selects the set of jobs the follower next schedules. Their approach failed finding the optimal solution already for several instances with $25$ jobs.
Following the three-field classification scheme (\cite{grahamOptimizationApproximationDeterministic1979} extended by \cite{tkindtSingleMachineAdversarial2024}), this problem is denoted by $1|OPT-n,r_j,\tilde{d_j}| \sum_j w_j^L R_j, \sum_j w_j^F C_j^F$, with $OPT-n$ meaning that the leader selects $n$ jobs in an optimistic setting. Moreover, the Boolean variable $R_j=1$ means that job $J_j$ is rejected by the leader, who minimizes the total cost of the rejected jobs $\sum_j w_j^L R_j$, after which the follower schedules the jobs in order to minimize the weighted sum of completion times. This problem was shown weakly $\cal{NP}$-hard \citep{kisBilevelMachineScheduling2012a}.
Furthermore, \cite{kisBilevelMachineScheduling2012a} studied the two versions \textit{optimistic} and \textit{pessimistic} of the bilevel scheduling problem: the $P|OPT-A_k|\sum_j^L C_j^L, \sum_j w_j^F C_j^F$ and $P|PES-A_k|\sum_j^L C_j^L, \sum_j w_j^F C_j^F$ problems, where $A_k$ is a set of jobs assigned to any machine $k$ by the leader, while the follower sequences jobs on each machine. These problems are shown to be strongly $\cal{NP}$-hard.
\cite{konurAnalysisDifferentApproaches2013a} studied scheduling of inbound trucks at the inbound doors of a cross-dock facility under truck arrival time uncertainty. A Genetic Algorithm was proposed to solve both \textit{optimistic} and \textit{pessimistic} bilevel problems. The leader's problem is the minimization of the expected total service time of the cross-dock operator while the follower's problem is the minimization or maximization, regarding the setting, of the expected total waiting time by determining the expected arrival times.
\cite{biancoGridSchedulingBilevel2015} proposed a bilevel scheduling problem occurring in grid computing. The leader assigns jobs to the follower aiming to minimize the total cost for rejecting or delaying tasks while the follower wants to maximize the computational resource usage efficiency. They also proposed an exact algorithm based on a bilevel programming formulation, as well as, a reformulation as a single-level Mixed Integer Programming (MIP) by using Karush-Kuhn-Tucker (KKT) optimality conditions. This MIP is often referred to as a Mathematical Program with Equilibrium Constraints (MPEC) \citep{dempeAnnotatedBibliographyBilevel2003}. Finally, the paper proposes a heuristic algorithm to solve large size instances.
\cite{tkindtSingleMachineAdversarial2024} studied several bilevel scheduling problems, with one machine and the adversarial setting. The paper considers jobs selection and data modification for the leader's decision and several usual objective functions from the scheduling literature. They provided some complexity proofs as well as exact polynomial-time algorithms. At the current state of the art, these problems appear very challenging and exact solution approaches can just handle instances very small in size.

This study focuses on a scenario with one leader and one follower composed of uniform parallel machines operating at two speeds. We consider a periodic decision-making framework with a time horizon $T$ to account for dynamic effects. The leader minimizes the weighted number of tardy jobs, whereas the follower's objective is to minimize the sum of completion times. From a bilevel perspective, we consider an optimistic case where the follower returns a schedule that achieves the minimum number of weighted tardy jobs among all optimal schedules for the total completion time.\\
Such a problem may occur in the context of the Industry of the Future in which the manufacturing system is composed of autonomous cyber-physical systems (CPS) interacting with each other in a hierarchical way. Thus, the problem is to dispatch production orders to these CPS. In such an organization, the leader is the system and the followers are the CPS. The integration of maintenance in the scheduling process may lead to setup the speed to potentially defective machines to its lowest value while healthy machines work at their fastest speed.
%also known as the Industry 4.0 (I4.0). I4.0 is an emerging industry that leverages technologies such as interconnected objects and cyber-physical systems (CPS). A manufacturing system in this context consists of multiple CPS interacting with each other. However, these CPS operate autonomously, making decisions based on their own interests. The integration of a centralized system (the manufacturing system) with decentralized systems (the CPS) allows for the incorporation of maintenance issues into the decision-making process. To illustrate this, consider that the leader has three options at the beginning of each scheduling period: (1) the speed of healthy machines is set to the maximum speed, (2) the speed of machines that are likely to be defective in a short time period, is set to the minimum speed to limit the risk of breakdowns, (3) defective machines are removed from the list of available resources for the follower and undergo maintenance.

Our contributions focus on two main aspects. First, we present several complexity results and, in particular, we establish that the considered bilevel scheduling problem is strongly
$\mathcal{NP}$-hard. 
To this extent, we introduce a MIP formulation with a polynomial number of variables and constraints that rules out the possibility
 that the bilevel problem belongs to the second level of the polynomial hierarchy \citep{woegingerTroubleSecondQuantifier2021} and provide a polynomial-time 
reduction from the Numerical 3-Dimensional Matching problem which is known to be 
$\mathcal{NP}$-complete in the strong sense. 
Next, we present a moderately exponential-time exact dynamic programming procedure and a dedicated branch-and-bound algorithm embedding a column generation scheme for the computation of the lower bound. Experimental results reveal that the bilevel problem is extremely challenging to solve and only instances 
with up to $80$ jobs and $4$ machines can be attacked.
The remainder of this paper is structured as follows. \href{sec:DefAndPPt}{Section 2},  presents structural properties and a MIP formulation as well as complexity results and a dynamic programming formulation. \href{sec:Exact_algorithm}{Section 3} focuses on the branch-and-bound algorithm. \href{sec:ExperRes}{Section 4}, reports on computational experiments conducted using randomly generated instances. Finally,  \href{sec:CCL}{Section 5} draws conclusions and outline potential directions for future research.

\section{Properties, modeling and complexity issues}
\label{sec:DefAndPPt}

We assume that the leader has taken his maintenance decisions, i.e., changed the floor shop configuration: we have $m_{1}$ machines with high speed denoted by set $\mathcal{V}_{1}$, and $m_{0}$ machines with low speed denoted by set $\mathcal{V}_{0}$. 
Correspondingly, let have $m=m_1+m_0$ machines. A set $\mathcal{J}$ of $N$ jobs is given among which $n \leqslant N$ have to be scheduled, with $n$ a given value. Each job $J_j$ is characterized by its processing time $p_j$, weight $w_j$, and due date $d_j$. Thus, the leader selects a subset $\mathcal{I} \subset \mathcal{J}$ of $n$ jobs and assigns them to the follower who, next, schedules them on the machines. The leader aims to minimize the number of weighted tardy jobs $\sum_j w_j U_j^L$ whereas the follower aims to minimize the sum of completion times $\sum_j C_j^F$. Let $\mathfrak{S}_{\mathcal{I}}$ be the set of machine schedules that contain $n$ jobs from $\mathcal{I}$, and let $C_j^F(s)$ be the completion time of job $J_j$ in machine schedule $s$.

Then, we define $U_j^L(s)=\begin{cases}
		1 \text{ if } C_j^F(s) > d_j \\
		0 \text{ otherwise}
	\end{cases}$. In the follower's problem, the optimistic scenario implies that there is a lexicographical objective function, since we want the schedule that minimizes the leader's objective function among all optimal schedules for the follower's objective function. Consequently, the follower minimizes $Lex \left( {\sum_j C_j^F}, {\sum_j w_jU_j^L} \right) $. The bilevel problem can be formulated as:

\begin{equation} \label{eq:bilevel_problem}
	\begin{aligned}
		 & \min_{\substack{\mathcal{I} \subset \mathcal{J}                                                                                                        \\ |\mathcal{I}|=n}} \hspace{2mm} \sum_{j \in \mathcal{I}} w_jU_j^L(s^*)  \\
		 & \text{st.}                                                                                                                                             \\[-6.5mm]
		 & \hspace{10mm} s^* \in \argmin_{s \in \mathfrak{S}_{\mathcal{I}}} \left( Lex \left( \sum_{j \in s} C_j^F(s), \sum_{j \in s} w_jU_j^L(s) \right) \right) \\
	\end{aligned}
\end{equation}

Correspondingly, we denote this bilevel problem by $Q|V_i \in \{V_{0},V_{1}\},OPT-n|\sum_j C_j^F,\sum_j w_jU_j^L$. The related follower's problem is denoted by $Q | V_i \in \{V_{0},V_{1}\}| Lex \left( \sum_j C_j,\sum_j w_jU_j \right)$. For the rest of the paper, we assume that all jobs are sorted by Shortest Processing Time (SPT) rule, i.e., $p_1 \leqslant p_2 \leqslant \hdots \leqslant p_N$. In case of equal-size jobs, i.e., jobs with same processing times, we sort them according to the Earliest Due Date (EDD) rule.

% --------------------------
%  -------------------------
% 
%	 Block structure	
% 
% -------------------------- 
% --------------------------

\subsection{Block structure}\label{subsec:BlockStruct}

In this section, we consider structural properties for the $Q || \sum_j C_j$ problem. We know from \cite{conwayTheoryScheduling1967b} that the problem can be solved in $\mathcal{O}(n \log n)$ time using the Shortest Processing Time job on the First Available Machine (STP-FAM) rule. The follower's problem is a special case where there are only two different speeds. This implies a special property for optimal schedules that is obtained by adapting the proof of \cite{conwayTheoryScheduling1967b}.

Given $n$ jobs and $m$ machines $M_i$ with a speed $V_i$, the processing time of a job $J_j$ on the machine $M_i$ is $p_j/V_i$. Our objective is to minimize $	\sum_j C_j$.
We represent a schedule by a non-rectangular and two-dimensional list that is a function of the machine number $i$ and of the chronological order $k$. We denote by $J_{[i,-\ell]}$ the $\ell^{\text{ th}}$ job on the machine $M_i$ from the end, with $p_{[i,-\ell]}$ its processing time. We also denote by $J_{[i,k]}$ the $k^{\text{ th}}$ job on the machine $M_i$ from the beginning, with $p_{[i,k]}$ its processing time.

A schedule $\sigma$ is given by :
$
	\sigma = \begin{pmatrix}
		J_{[1,1]} & \hdots & J_{[1,\eta_{1}]} \\
		          & \hdots &                  \\
		J_{[m,1]} & \hdots & J_{[m,\eta_m]}   \\
	\end{pmatrix}
$
where $\eta_i$ denotes the number of jobs assigned on machine $M_i$, with $\sum_{i=1}^{m} \eta_i = n$.

Since all $\eta_i$ depend on the schedule $\sigma$, we can have multiple values of $\eta_i$. To simplify the notation, let us introduce some dummy jobs with a processing time and a weight equal to zero. We will ensure that each machine is assigned exactly $n$ jobs, including the dummy ones. It is clear that in an optimal schedule, these dummy jobs are scheduled on the left and can be removed without modifying the objective function nor the optimality of the schedule. To facilitate notation, we will denote all dummy jobs with the index $0$, thus, $p_0=0$. Thus we add $(mn-n)$ dummy jobs. This results in the modified schedule $\sigma$ being defined as :
\begin{equation}\nonumber \sigma = \scalebox{0.85}{$
			\left(\begin{array}{ccccccc}
				\multirow{2}{*}{0} & \overbrace{\hdots}^{n-\eta_{1} \hspace{2mm} times} & \multirow{2}{*}{0} &        & \multirow{2}{*}{$J_{[1,n-\eta_{1}+1]}$} & \multirow{2}{*}{$\hdots$} & \multirow{2}{*}{$J_{[1,n]}$} \\[5mm]
				                   &                                                    &                    & \hdots &                                         &                           &                              \\
				\multirow{2}{*}{0} & \overbrace{\hdots}^{n-\eta_m \hspace{2mm} times}   & \multirow{2}{*}{0} &        & \multirow{2}{*}{$J_{[m,n-\eta_m+1]}$}   & \multirow{2}{*}{$\hdots$} & \multirow{2}{*}{$J_{[m,n]}$} \\
			\end{array}\right)$ }
	=
	\begin{pmatrix}
		J_{[1,1]} & \hdots & J_{[1,n]} \\
		          & \hdots &           \\
		J_{[m,1]} & \hdots & J_{[m,n]} \\
	\end{pmatrix}
\end{equation}

We focus on a ranking from right to left, i.e., in anti-chronological order $\ell=n+1-k$ (with $\ell=1$ for the last, 2 for the second-to-last job, and so forth). Since the completion time of the job on machine $M_i$ at position $k$ can be expressed as $C_{[i,k]}=\sum\limits_{h=1}^k p_{[i,h]}/V_i$, we see that $\sum\limits_{k=1}^n C_{[i,k]} = \sum\limits_{k=1}^n \sum\limits_{h=1}^k p_{[i,h]}/V_i = \sum\limits_{h=1}^n (n+1-h) p_{[i,h]}/V_i$. Hence, we find that $\sum\limits_{j=1}^n C_j = \sum\limits_{\substack{1\leqslant i \leqslant m \\ 1 \leqslant \ell \leqslant n}} \ell/V_i \times p_{[i,-\ell]}$

Correspondingly, we must minimize the weighted sum of processing times $p_{[i,-\ell]}$ with positional weights $\omega_{[i,\ell]} = \ell/V_i$, for position $\ell$ $(\ell=1,\hdots,n)$ from the end on machine $M_i$. The $\omega_{[i,\ell]}$'s do not depend on the jobs or on the schedule. It is well known that such a sum of pairwise products of two sequences of reals is achieved if one sequence is arranged in increasing order and the other in decreasing order \citep{hardyInequalitiesCambridgeMathematical1934}. To get the optimal solution, we have to sort the $\omega_{[i,\ell]}$'s in non-decreasing order and the jobs in the reverse order, i.e. Longest Processing Time (LPT) rule. Specifically, the $k^{\text{th}}$ smallest coefficient $\omega_{[k]}$ is associated with the $k^{\text{th}}$ largest processing time $p_{[k]}$. Consequently, we have the following theorem.

\begin{theorem}\label{thm:Solve-Q|Vo,Vmax|Sum-Cj}
	The $Q|V_i \in \{V_{0},V_{1}\}|\sum_j C_j$ problem is solvable in $\mathcal{O}(n \log n)$ time.
\end{theorem}

\qed

\begin{remark}
	Firstly, the proof of the general case, i.e., $Q||\sum_j C_j$, provides the algorithm by constructing the schedule from the right to the left, i.e., from the end of the machine schedule to the beginning. Alternatively, we can equivalently consider an SPT-FAM rule, thus scheduling jobs from the left to the right.

	Secondly, sets with identical positional weights correspond to batches of jobs. Furthermore, when there are only two machine speeds involved, an optimal schedule is characterized by a repetition of blocks. Figure \ref{fig:schedulingPatternFProb} illustrates the blocks structure, where jobs within the dashed group signify their membership to a block and can subsequently undergo permutation without changing the value of $\sum_j C_j$. We denote such a block as $\mathcal{B}_{b}$, which is a set composed of pairs $(i,k)$ where $i$ is the machine index and $k$ is the position on the machine schedule.

\end{remark}

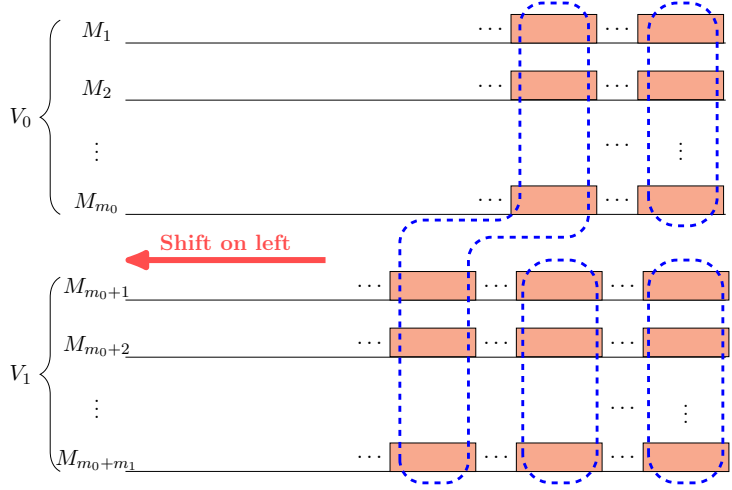
\begin{figure}[h]
	\centering
	\scalebox{0.76}{%
		\begin{tikzpicture}

			\def\yVmax{-4.5} % the y value for draw Vmax machines

			% draw machines V0    
			\draw[] (0,0) node {$M_1$};
			\draw[ultra thin] (0.5,-0.2) -- (11,-0.2);
			\draw[] (0,-1) node {$M_2$};
			\draw[,ultra thin] (0.5,-1.2) -- (11,-1.2);
			\draw[] (0,-2) node {$\vdots$};
			\draw[] (0,-3) node {$M_{m_{0}}$};
			\draw[ultra thin] (0.5,-3.2) -- (11,-3.2);
			\draw [decorate,decoration={brace,amplitude=10pt,mirror,raise=4pt}] (-0.5,0.2) -- node[left=5mm]{$V_{0}$} (-0.5,-3.2);

			% draw machines Vmax
			\draw[] (0,\yVmax) node {$M_{m_{0}+1}$};
			\draw[ultra thin] (0.5,\yVmax-0.2) -- (11,\yVmax-0.2);
			\draw[] (0,\yVmax-1) node {$M_{m_{0}+2}$};
			\draw[ultra thin] (0.5,\yVmax-1.2) -- (11,\yVmax-1.2);
			\draw[] (0,\yVmax-2) node {$\vdots$};
			\draw[] (0,\yVmax-3) node {$M_{m_{0} + m_{1}}$};
			\draw[ultra thin] (0.5,\yVmax-3.2) -- (11,\yVmax-3.2);
			\draw [decorate,decoration={brace,amplitude=10pt,mirror,raise=4pt}] (-0.5,\yVmax+0.2) -- node[left=5mm]{$V_{1}$} (-0.5,\yVmax-3.2);

			% draw jobs on V0
			\matrix [anchor=west,ampersand replacement=\&] at (6.4,-1.45)
			{
				\node {$\hdots$};
				\&\node [rectangle,draw,fill=Red!40,minimum width = 15mm,minimum height = 5mm] {};
				\& \node {$\hdots$};
				\&\node [rectangle,draw,fill=Red!40,minimum width = 15mm,minimum height = 5mm] {}; \\[4.8mm]

				\node {$\hdots$};
				\&\node [rectangle,draw,fill=Red!40,minimum width = 15mm,minimum height = 5mm] {};
				\&\node {$\hdots$};
				\&\node [rectangle,draw,fill=Red!40,minimum width = 15mm,minimum height = 5mm] {}; \\[4mm]

				\node {};
				\& \node {};
				\& \node {$\hdots$};
				\& \node[inner xsep=7mm] {$\vdots$};                                               \\[3.1mm]

				\node {$\hdots$};
				\&\node [rectangle,draw,fill=Red!40,minimum width = 15mm,minimum height = 5mm] {};
				\&\node {$\hdots$};
				\&\node [rectangle,draw,fill=Red!40,minimum width = 15mm,minimum height = 5mm] {}; \\
			};

			% draw jobs on Vmax

			\matrix [anchor=east,ampersand replacement=\&] at (11.2,\yVmax-1.45)
			{
				\node {$\hdots$};
				\&\node [rectangle,draw,fill=Red!40,minimum width = 15mm,minimum height = 5mm] {};
				\& \node {$\hdots$};
				\&[-0.1mm]\node [rectangle,draw,fill=Red!40,minimum width = 15mm,minimum height = 5mm] {};
				\& \node {$\hdots$};
				\&\node [rectangle,draw,fill=Red!40,minimum width = 15mm,minimum height = 5mm] {}; \\[4.8mm]

				\node {$\hdots$};
				\&\node [rectangle,draw,fill=Red!40,minimum width = 15mm,minimum height = 5mm] {};
				\& \node {$\hdots$};
				\&\node [rectangle,draw,fill=Red!40,minimum width = 15mm,minimum height = 5mm] {};
				\& \node {$\hdots$};
				\&\node [rectangle,draw,fill=Red!40,minimum width = 15mm,minimum height = 5mm] {}; \\[5mm]

				\node {};
				\& \node {};
				\& \node {};
				\& \node {};
				\& \node {$\hdots$};
				\& \node[inner xsep=7mm] {$\vdots$};                                               \\[2.1mm]

				\node {$\hdots$};
				\&\node [rectangle,draw,fill=Red!40,minimum width = 15mm,minimum height = 5mm] {};
				\& \node {$\hdots$};
				\&\node [rectangle,draw,fill=Red!40,minimum width = 15mm,minimum height = 5mm] {};
				\& \node {$\hdots$};
				\&\node [rectangle,draw,fill=Red!40,minimum width = 15mm,minimum height = 5mm] {}; \\
			};

			% draw permutation jobs Vmax

			% 1st block
			\draw [blue,dashed,line width=0.5mm]
			(9.65,\yVmax -3.4) [rounded corners=5mm] -- (9.65,\yVmax+0.5)
			-- (10.95,\yVmax+0.5) [rounded corners=5mm] -- (10.95,\yVmax -3.4) [rounded corners=5mm] -- cycle;

			% 2nd block
			\draw [blue,dashed,line width=0.5mm]
			(7.45,\yVmax -3.4) [rounded corners=5mm] -- (7.45,\yVmax+0.5)
			-- (8.75,\yVmax+0.5) [rounded corners=5mm] -- (8.75,\yVmax -3.4) [rounded corners=5mm] -- cycle;

			% draw permutation jobs V0 

			% 1st block
			\draw [blue,dashed,line width=0.5mm]
			(9.65,-3.4) [rounded corners=5mm] -- (9.65,0.5)
			-- (10.85,0.5) [rounded corners=5mm] -- (10.85,-3.4) [rounded corners=5mm] -- cycle;

			% draw permutation jobs V0 and Vmax
			% 1th column

			\draw [blue,dashed,line width=0.5mm]
			(5.3,\yVmax -3.4) [rounded corners=4mm] -- (5.3,\yVmax+1.2) [rounded corners=4mm] -- (7.4,\yVmax+1.2) [rounded corners=4mm] -- (7.4,0.5) [rounded corners=4mm] -- (8.6,0.5) -- (8.6,\yVmax+0.9) [rounded corners=4mm] -- (6.5,\yVmax+0.9) [rounded corners=4mm] -- (6.5,\yVmax -3.4) [rounded corners=4mm] -- cycle;

			% draw caler à gauche
			\draw[line width=3pt,red!70,{Latex[round]}-,] (0.5,\yVmax+0.5) -- (4,\yVmax+0.5) node [midway,above] {\textbf{Shift on left}};

		\end{tikzpicture}
	}
	\caption{Scheduling pattern with block structure for the solution of $Q|V_i \in \{V_{1},V_{0}\}|\sum_j C_j$}
	\label{fig:schedulingPatternFProb}
\end{figure}

This is an important result because it enables us to caracterize the set of optimal schedules for the first criterion ($\sum_j C_j$). Thus, the problem is to determine which permutation of jobs within a block should be applied to obtain an optimal schedule for the second criterion ($\sum_j w_j U_j$).

% --------------------------
%  -------------------------
% 
%	 ADV	
% 
% -------------------------- 
% --------------------------

\subsection{Mixed Integer Programming}\label{subsec:mipmodel}

In this subsection we show that the bilevel problem can be modeled by means of Mixed Integer Programming. Before presenting the model, we introduce a few notations.

We have $b_{\max}$ blocks $\mathcal{B}_b$, $1 \leqslant b \leqslant b_{\max}$, given by the block structure described in Section~\ref{subsec:BlockStruct}. Let $E=\left\{(i,k)\right\}$ be a set of locations (just like $\mathcal{B}_b$) where $i$ is the machine index and $k$ the position on the machine schedule in chronological order; we have $E=\mathcal{B}_1 \cup \hdots \cup \mathcal{B}_{b_{\max}}$.
We need to assign a job to each position in the blocks $\mathcal{B}_2, \hdots, \mathcal{B}_{b_{\max}}$, but the first block $\mathcal{B}_1$ may not be completely filled. However, we can compute the number of jobs $Q$ that should go into $\mathcal{B}_1$ as $Q=n- \sum_{b=2}^{b_{\max}} |\mathcal{B}_b|$.

We use binary variables $x_{i,j,k}$, such that $x_{i,j,k}=1$ if job $J_j$ is scheduled at location $(i,k)$; and zero otherwise. We also use binary variables $U_{i,j,k}$, with $U_{i,j,k}=1$ if job $J_j$ is scheduled at location $(i,k)$ and is penalized. The equations only state that tardiness implies penalty. So, a job $J_j$ may be a priori unduly penalized, but this situation is, at the end, avoided by the minimization of penalties. So, we have $U_j \leqslant \sum_{(i,k)\in E} U_{i,j,k}$, and the minimization of $\sum_j \sum_{(i,k)\in E} w_j \, U_{i,j,k}$ leads to minimizing $\sum_j w_j U_{j}$. We use variables $C_{i,k}$ to express the completion time of the job in position $k$ on machine $M_i$.

The following MIP model holds:

{
\setstretch{0}
\begin{flalign}
	 & \min \sum_{j=1}^N \sum_{(i,k)\in E} w_j U_{i,j,k} \label{eq:objMIP}                                                                                                                                                                                                              \\
	 & \textbf{s.t.}                                                                                  &  & \nonumber                                                                                                                                                                    \\
	 & \sum_{(i,k)\in E} x_{i,j,k} \leqslant 1                                                        &  & \forall j \in \{1,\hdots,N\} \label{eq:noDuplicationJob}                                                                                                                     \\
	 & \sum_{j=1}^N x_{i,j,k} \leqslant 1                                                             &  & \forall (i,k) \in E \label{eq:disjonctivesConstLoc}                                                                                                                          \\
	 & \sum_{j=1}^N \sum_{(i,k)\in E \setminus \mathcal{B}_1} x_{i,j,k} = n-Q                         &  & \label{eq:fillAllBlock}                                                                                                                                                      \\
	 & \sum_{j=1}^N \sum_{(i,k)\in \mathcal{B}_1} x_{i,j,k} = Q \label{eq:tryFillFirstBlock}                                                                                                                                                                                            \\
	 & \sum_{j=1}^N x_{i,j,k} \, p_j \leqslant \sum_{j=1}^N x_{i',j,k'} \, p_j                        &  & \forall 1 \leqslant b \leqslant b_{\max}-1 \, , \forall (i,k) \in \mathcal{B}_b \, , \forall (i',k') \in \mathcal{B}_{b+1}\label{eq:incrPjByBlock}                           \\
	 & C_{i,k} = C_{i,k-1} + \sum_{j=1}^N x_{i,j,k} p_j/V_i                                   &  & \forall (i,k) \in E                                                                                                                                &  & \label{eq:computeCj} \\
	 & C_{i,0}=0                                                                                      &  & \forall i \in \{1,\hdots,m\}                                                                                                                                                 \\
	 & U_{i,j,k} \leqslant x_{i,j,k}                                                                  &  & \forall j \in \{1,\hdots,N\}\, , \forall (i,k) \in E \label{eq:lateImplScheduled}                                                                                            \\
	 & C_{i,k} \leqslant \sum_{j=1}^N x_{i,j,k} \, d_j + \left(\sum_{j=1}^N U_{i,j,k}\right)\, H(i,k) &  & \forall (i,k) \in E \label{eq:checkDueDate}
\end{flalign}
}
Here $H(i,k) = k/V_i \times \max\limits_j p_j$, which is an upper bound on the completion time of the $k$th job on machine $M_i$.

The objective function (\ref{eq:objMIP}) minimizes the weighted number of tardy jobs. Constraints (\ref{eq:noDuplicationJob}) ensure that no jobs are duplicated. The disjunctive constraints (\ref{eq:disjonctivesConstLoc}) ensure that each location has at most one job. Constraints (\ref{eq:fillAllBlock}) and (\ref{eq:tryFillFirstBlock}) guarantee that all blocks, except maybe the first one, must be completely filled: these constraints imply that we select exactly $n$ jobs. Constraints (\ref{eq:incrPjByBlock}) enforce that the processing times of assigned jobs between consecutive blocks do not decrease. Constraints (\ref{eq:computeCj}) allow to compute the completion times. Constraints (\ref{eq:lateImplScheduled}) ensure that only a scheduled job can be penalized. Finally, constraints (\ref{eq:checkDueDate}) guarantee that a job that is not penalized is necessarily early, and so, we need to check that its completion time is smaller than its due date.

This mathematical formulation has $\mathcal{O}(N(n+m))$ variables and $\mathcal{O}(N^2+N m^2)$ constraints. 

\subsection{Complexity results}

\subsubsection{The follower's problem}\label{subsubsec:FollowerPb}

Firstly, we demonstrate that in the case of equal-size jobs, the follower's problem, with an arbitrary number of machine speeds can be solved in polynomial time.

\begin{theorem}\label{thm:Solve-Q|pj=p|Lex(Sum-Cj,Sum-wj)}
	The $Q | p_j=p | Lex \left( \sum_{j} C_j, \sum_{j} w_jU_j \right)$ problem is polynomially solvable.
\end{theorem}

\begin{proof}
	We know that all optimal schedules for the $ Q||\sum_{j} C_j$ problem are characterized by the block structure, and the SPT-FAM rule gives an optimal solution \citep{conwayTheoryScheduling1967b} in $\mathcal{O}(n \log n)$ time. Thus, we can compute all completion times, as jobs' processing times are equal. Moreover, by extending the algorithm mentioned by \cite{dessoukySchedulingIdenticalJobs1990a}, we can solve this problem. The idea is to determine which job is on time if assigned to a completion time, and then choose among them the one with the maximum weight. To do this, the jobs are initially sorted according to the reverse EDD order, in $\mathcal{O}(n \log n)$ time. Then, we start with the largest completion time and assign the early job with the largest weight. If no early job exists then this completion time is temporarily left unassigned, and we move to the second largest completion time, repeating the above operation. At the end of this first pass, when all completion times have been considered, then the remaining jobs can be randomly assigned to the remaining unassigned completion times: they are all tardy and with minimum weight. This can be implemented to run in $\mathcal{O}(n \log n)$ time.
\end{proof}

Now, we focus on the general case with arbitrary processing times. In the case of $2$ identical machines, we show that the follower's problem is $\mathcal{NP}$-hard in the weak sense.

\begin{theorem}\label{thm:P2||Lex(SumCj,SumUj)-m-fixed-NP-hard}
	The $P2||Lex \left( \sum_j C_j,\sum_j U_j \right)$ problem is $\mathcal{NP}$-hard in the weak sense.
\end{theorem}

\begin{proof}
	Let us consider the following problem which is known to be $\mathcal{NP}$-complete in the weak sense \citep{Garey1988}:

	\texttt{Even-Odd Partition (EOD)}

	\begin{tabular}{p{0.1\textwidth}p{0.9\textwidth}}
		\underline{Data:}     & $2n$ positive integers $a_1 < a_2 < \hdots < a_{2n}$ where $A = 1/2 \times \sum_{j=1}^{2n}a_j$. \\
		\underline{Question:} & Does there exist a partition into two sets $A_1$ and $A_2$, each containing                      \\ & exactly one element from the pair $\{ a_{2j-1},a_{2j} \}$ for $j=1, \hdots ,n$, such that \\ & $\sum_{j \in A_{1}} a_j = \sum_{j \in A_2} a_j = A$? \\
	\end{tabular}

	We show that this problem reduces to the decision version of problem $P2||Lex \left( \sum_j C_j,\sum_j U_j \right)$:

	\texttt{DEC}

	\begin{tabular}{p{0.1\textwidth}p{0.8\textwidth}}
		\underline{Data:}     & A set of jobs $J_j$ with a processing time $p_j$, a weight $w_j$ and a due date $d_j$, to be processed on 2 identical machines. Let $C^{\star}$ be the value of an optimal solution of $P2||\sum_j C_j$, and $U$ a given value. \\
		\underline{Question:} & Does there exist a schedule $\sigma$ such that $\sum_j C_j(\sigma) = C^{\star}$ and $\sum_j U_j(\sigma) \leq U$?                                                                                                                \\
	\end{tabular}

	We show that there exists a reduction $\propto$ from \texttt{EOD} to \texttt{DEC}. Let $I$ be a given instance of \texttt{EOD} and $\propto(I)$ is the corresponding instance of \texttt{DEC}. This instance $\propto(I)$ contains $2n$ jobs that must be scheduled on 2 identical machines. Job $J_j$ corresponds to $a_j$. The processing times are defined as follows: $\forall k=0, \hdots,n-1$ we put
	$p_{2k+1} = 2kA + a_{2k+1}$, $p_{2k+2} = 2kA + a_{2k+2}$. Furthermore, each job $J_j$ $(j=1, \ldots, 2n)$ has due date $d_{j} =D  = 1/2 \times \sum_{i=1}^{2n}p_i=\sum_{k=0}^{n-1} \left(p_{2k+1}/2+p_{2k+2}/2\right)=\sum_{k=0}^{n-1} 2A + A=n(n-1)A + A$. Finally, we put $U=0$.

	We demonstrate that the \texttt{EOD} problem's answer is true if and only if there exists a schedule $\sigma$ that has $\sum_j C_j(\sigma) = C^{\star}$ and in which all jobs are on time, i.e., $\sum_j U_j (\sigma )=0$.

	We show the forward implication ($\Rightarrow$). Let us consider the scenario where the answer to the \texttt{EOD} problem is true. This implies that there exists a partition in two sets, $A_1$ and $A_2$, each containing exactly one element from the pair $\{ a_{2j-1},a_{2j} \}$ for $j=1, \hdots, n$, such that $\sum_{j \in A_{1}} a_j = \sum_{j \in A_2} a_j = A$. We create our schedule by putting $J_j$ on $M_1$ if $j \in A_1$, and on $M_2$, otherwise. Consequently, the completion times of the last job on each machine are $C_{M_1}=\sum_{j \in A_1} p_j = \sum_{j \in A_1} 2 \left\lfloor (j-1)/2 \right\rfloor A + a_j = 2A\sum_{k = 0}^{n-1} k + A= n(n-1)A + A$ and similarly $C_{M_2}=\sum_{j \in A_2} p_j =n(n-1)A + A$. Since the common due date equals $n(n-1)A + A$, this indicates that all jobs can, indeed, be scheduled on time and $\sum_j C_j$ is minimum as this schedule follows the SPT-FAM rule (see Figure \ref{fig:machineSchedulesP2||Lex(Sum-Cj,Sum-wj)}).
	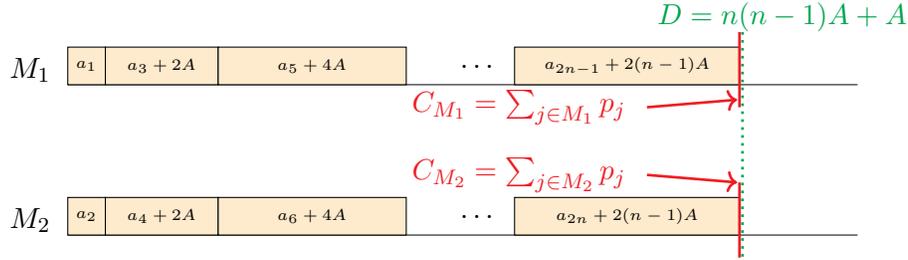
\begin{figure}[h]
		\centering
		\scalebox{1}{%
			\begin{tikzpicture}
				% draw machines
				\draw (0,0) node {$M_1$};
				\draw[ultra thin] (0.5,-0.2) -- (11,-0.2);
				\draw (0,-2) node {$M_2$};
				\draw[ultra thin] (0.5,-2.2) -- (11,-2.2);

				% draw jobs on M1 machine
				\draw (0.75,0.05) node [rectangle,draw,fill=YellowOrange!20,minimum width = 5mm,minimum height = 5mm,node font=\tiny] {$a_1$};

				\draw (1.75,0.05) node [rectangle,draw,fill=YellowOrange!20,minimum width = 15mm,minimum height = 5mm,node font=\tiny] {$a_3 + 2A$};

				\draw (3.75,0.05) node [rectangle,draw,fill=YellowOrange!20,minimum width = 25mm,minimum height = 5mm,node font=\tiny] {$a_5+ 4 A$};

				\draw (5.94,0.05) node [rectangle,minimum height = 5mm] {$\hdots$};

				\draw (7.94,0.05) node [rectangle,draw,fill=YellowOrange!20,minimum width = 30mm,minimum height = 5mm,node font=\tiny] {$a_{2n-1} + 2(n-1) A$};

				% draw jobs on M2 machine
				\draw (0.75,-1.95) node [rectangle,draw,fill=YellowOrange!20,minimum width = 5mm,minimum height = 5mm,node font=\tiny] {$a_2$};

				\draw (1.75,-1.95) node [rectangle,draw,fill=YellowOrange!20,minimum width = 15mm,minimum height = 5mm,node font=\tiny] {$a_4 + 2A$};

				\draw (3.75,-1.95) node [rectangle,draw,fill=YellowOrange!20,minimum width = 25mm,minimum height = 5mm,node font=\tiny] {$a_6+ 4A$};

				\draw (5.94,-1.95) node [rectangle,minimum height = 5mm] {$\hdots$};

				\draw (7.94,-1.95) node [rectangle,draw,fill=YellowOrange!20,minimum width = 30mm,minimum height = 5mm,node font=\tiny] {$a_{2n} + 2(n-1)A$};

				% draw completion times
				\draw[line width=1pt,Red] (9.43,0.5) -- (9.43,-0.5);
				\draw[Red] (6.5,-0.5) node {$C_{M_1}=\sum_{j \in M_1} p_j$};
				\draw[->,Red,line width=1pt] (8.2,-0.5) -- (9.4,-0.4);

				\draw[Red,line width=1pt] (9.43,-1.5) -- (9.43,-2.5);
				\draw[Red] (6.5,-1.4) node {$C_{M_2}=\sum_{j \in M_2} p_j$};
				\draw[->,Red,line width=1pt] (8.2,-1.4) -- (9.4,-1.5);

				% draw due date
				\draw[line width=1pt, dotted,Green] (9.47,0.5) -- (9.47,-2.5);

				\draw[Green] (10,0.7) node {$D=n(n-1)A+A$};
			\end{tikzpicture}
		}
		\caption{Machine schedules for the problem $P2||Lex \left( \sum_j C_j,\sum_j U_j \right)$}
		\label{fig:machineSchedulesP2||Lex(Sum-Cj,Sum-wj)}
	\end{figure}

	Now suppose that there exists a schedule $\sigma$ with $\sum_j C_j(\sigma) = C^{\star}$ and $\sum U_j (\sigma ) \leq U=0$. Because of the choice of the processing times, we know that the $k$th position (counted from the start) on machines $M_1$ and $M_2$ is occupied by the jobs $J_{2k-1}$ and $J_{2k}$. Moreover, since the total processing time of all jobs is equal to $2D$ and each job is on time, both machines must finish exactly at time $D$. It is easily verified that the desired sets $A_1$ and $A_2$ can be obtained as the sets containing the indices of the jobs executed by machines $M_1$ and $M_2$ in $\sigma$, respectively.

\end{proof}

We also show that when the number of machines is a part of the instance, then the follower's problem is $\mathcal{NP}$-hard in the strong sense.

\begin{theorem}\label{thm:P||LexSumCj,SumUj-NP-hard}
	The $P||Lex \left( \sum_j C_j,\sum_j U_j \right)$ problem is $\mathcal{NP}$-hard in the strong sense.
\end{theorem}

\begin{proof}
	Let us consider the following problem which is known to be $\mathcal{NP}$-complete in the strong sense \citep{GareyJohnson}:

	\texttt{Numerical 3-Dimensional Matching (NUM-3DM)}

	\begin{tabular}{p{0.1\textwidth}p{0.9\textwidth}}
		\underline{Data:}     & Three disjoint sets of integers $X=\{x_1, \ldots , x_n\}$, $Y=\{y_1, \ldots , y_n\}$, and $Z=\{z_1, \ldots , z_n\}$ along with a bound $b$.                                                                                                                           \\
		\underline{Question:} & Can $X \cup Y \cup Z$ be partitioned into $n$ disjoint sets $A_1, A_2, \hdots, A_n$ such that each $A_h$ $(h=1, \ldots ,n)$ contains exactly one element from each of $X$, $Y$, and $Z$, i.e., $A_h=\{x_{[h]},y_{[h]},z_{[h]}\}$ and for $1 \leqslant h \leqslant n$,
		$x_{[h]}+y_{[h]}+z_{[h]}=b$?                                                                                                                                                                                                                                                                  \\
	\end{tabular}

	We show that this problem reduces to the decision version of problem $P||Lex \left( \sum_j C_j,\sum_j U_j \right)$:
	\\
	\texttt{DEC}

	\begin{tabular}{p{0.1\textwidth}p{0.8\textwidth}}
		\underline{Data:}     & A set of jobs $J_j$ with a processing time $p_j$ and a due date $d_j$, to be processed on $m$ identical machines. Let $C^{\star}$ be the value of an optimal solution of $P||\sum_j C_j$ and $U$ a given value. \\
		\underline{Question:} & Does there exist a schedule $\sigma$ such that $\sum_j C_j(\sigma) = C^{\star}$ and $\sum_j U_j(\sigma) \leq U$?                                                                                                \\
	\end{tabular}

	We show that there exists a reduction $\propto$ from \texttt{NUM-3DM} to \texttt{DEC}. Let $I$ be a given instance of \texttt{NUM-3DM} and let $\propto(I)$ be the corresponding instance of \texttt{DEC}. We have: $U=0$ and given three multisets of integers $X$, $Y$, and $Z$, each containing $n$ elements, we create $3n$ jobs with the following processing times $\forall i=1, \hdots,n:$

	$p_{i} = x_i$ , $p_{n+i} = y_i + \sum_{x \in X} x$ , $p_{2n+i} = z_i + \sum_{y \in Y} y +\sum_{x \in X} x$ and a common due date \\$D = b + \sum_{y \in Y} y +2 \sum_{x \in X} x$.
		Furthermore, we introduce $m=n$ machines.

		Since we have $\forall i,j,k=1, \hdots,n \; , p_i < p_{n+j} < p_{2n+k}$, $\sum_j C_j$ is minimized by scheduling exactly one triplet $(x_{[h]}, y_{[h]}, z_{[h]})= (x_{i}, y_{j}, z_{k})$, with $i,j,k=1, \hdots,n$, on each machine $M_h$ $(h=1, \ldots , n)$ (see Figure \ref{fig:machineSchedulesP||Lex(Sum-Cj,Sum-wj)}).
		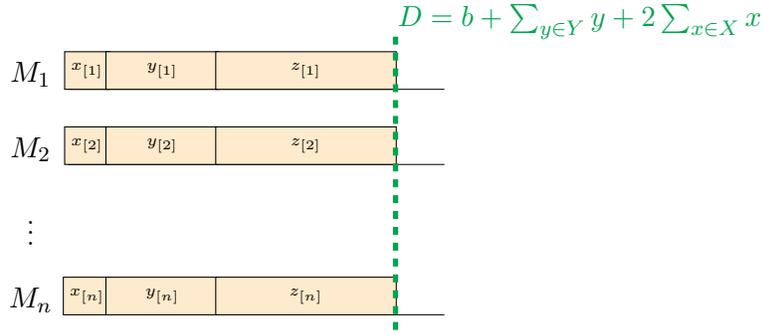
\begin{figure}[h]
			\centering
			\scalebox{1}{%
				\begin{tikzpicture}
					% draw machines
					\draw (0,0) node {$M_1$};
					\draw[ultra thin] (0.5,-0.2) -- (5.5,-0.2);
					\draw (0,-1) node {$M_2$};
					\draw[ultra thin] (0.5,-1.2) -- (5.5,-1.2);
					\draw (0,-2) node {$\vdots$};
					\draw (0,-3) node {$M_{n}$};
					\draw[ultra thin] (0.5,-3.2) -- (5.5,-3.2);

					% draw jobs on M1 machine
					\draw (0.75,0.05) node [rectangle,draw,fill=YellowOrange!20,minimum width = 5mm,minimum height = 5mm,node font=\tiny] {$x_{[1]}$};

					\draw (1.75,0.05) node [rectangle,draw,fill=YellowOrange!20,minimum width = 15mm,minimum height = 5mm,node font=\tiny] {$y_{[1]}$};

					\draw (3.66,0.05) node [rectangle,draw,fill=YellowOrange!20,minimum width = 24mm,minimum height = 5mm,node font=\tiny] {$z_{[1]}$};

					% draw jobs on M2 machine
					\draw (0.75,-0.95) node [rectangle,draw,fill=YellowOrange!20,minimum width = 5mm,minimum height = 5mm,node font=\tiny] {$x_{[2]}$};

					\draw (1.75,-0.95) node [rectangle,draw,fill=YellowOrange!20,minimum width = 15mm,minimum height = 5mm,node font=\tiny] {$y_{[2]}$};

					\draw (3.66,-0.95) node [rectangle,draw,fill=YellowOrange!20,minimum width = 24mm,minimum height = 5mm,node font=\tiny] {$z_{[2]}$};

					% draw jobs on Mn machine
					\draw (0.75,-2.95) node [rectangle,draw,fill=YellowOrange!20,minimum width = 5mm,minimum height = 5mm,node font=\tiny] {$x_{[n]}$};

					\draw (1.75,-2.95) node [rectangle,draw,fill=YellowOrange!20,minimum width = 15mm,minimum height = 5mm,node font=\tiny] {$y_{[n]}$};

					\draw (3.66,-2.95) node [rectangle,draw,fill=YellowOrange!20,minimum width = 24mm,minimum height = 5mm,node font=\tiny] {$z_{[n]}$};

					% draw due date
					\draw[line width=2pt, dashed,Green] (4.85,0.5) -- (4.85,-3.5);

					\draw[Green] (7.3,0.7) node {$D=b + \sum_{y \in Y} y +2 \sum_{x \in X} x$};
				\end{tikzpicture}
			}
			\caption{Machine schedules for the $P||Lex \left( \sum_j C_j,\sum_j U_j \right)$}
			\label{fig:machineSchedulesP||Lex(Sum-Cj,Sum-wj)}
		\end{figure}

		We demonstrate that the NUM-3DM problem's answer is true if and only if all jobs can be scheduled in time.

		Let us demonstrate the forward implication ($\Rightarrow$). Assuming the solution to the NUM-3DM problem is true implies the existence of $n$ disjoint triples $A_{h} = \{x_i,y_j,z_k\}$ such that $x_i+y_j+z_k=b$. Consequently, we schedule the jobs $\{p_i,p_{n+j},p_{2n+k}\}$ on the machine $M_{h}$. The completion time of the last job on machine $M_{h}$ is:

	$C_{M_h} = x_i+y_j+z_k+\sum_{y \in Y} y +2 \sum_{x \in X} x = b+\sum_{y \in Y} y +2 \sum_{x \in X} x=D$ ensuring all jobs are scheduled in time.
		Now, we establish the reverse implication ($\Leftarrow$).
		Suppose that there exists a schedule $\sigma$ such that $\sum_j C_j(\sigma) = C^{\star}$ and $\sum_j U_j(\sigma) \leq U=0$. Because of the choice of the processing times, we see that each machine $M_{h}$ $(h=1, \ldots , n)$ executes exactly one job from the sets $\{J_1, \ldots , J_n\}$, $\{J_{n+1}, \ldots , J_{2n}\}$, and $\{J_{2n+1}, \ldots , J_{3n}\}$. Furthermore, since the total processing time of all jobs is equal to $nD$ and all jobs are on time, all machine must finish exactly at time $D$. A feasible solution to the NUM-3DM problem is immediately obtained by constructing triplet $h$ $(h=1, \ldots ,n)$ according to the indices of the jobs executed by machine $M_h$.
\end{proof}

We can also derive the following result.

\begin{corollary}\label{coro:P||LexSumCj,SumUj-unapprox}
	The $P||Lex \left( \sum_j C_j,\sum_j U_j \right)$ problem is not approximable in polynomial or pseudo-polynomial time unless $\cal{P} = \cal{NP}$.
\end{corollary}

\begin{proof}
	This result directly follows from the proof of Theorem \ref{thm:P||LexSumCj,SumUj-NP-hard} where it is shown that the particular problem when $\sum_j U_j^{\star}=0$ is strongly $\cal{NP}$-hard. This rules out the existence of an approximation algorithm for ratio $\rho=\frac{\sum_j U_j}{\sum_j U_j^{\star}}$
\end{proof}
From Theorem \ref{thm:P||LexSumCj,SumUj-NP-hard} and Corollary \ref{coro:P||LexSumCj,SumUj-unapprox}, the following result immediately holds (as the considered problem is a generalization of the $P||Lex \left( \sum_j C_j,\sum_j U_j \right)$ problem).

\begin{corollary}\label{corol:followerProblemComplexity}
	The $Q | V_i \in \{V_{0},V_{1}\}| Lex \left( \sum_j C_j,\sum_j w_jU_j \right)$ problem is $\mathcal{NP}$-hard in the strong sense and is not approximable in polynomial or pseudo-polynomial time unless $\cal{P} = \cal{NP}$.
\end{corollary}

\subsubsection{The bilevel problem} \label{sec:Def-and-Ppt-Bilevel-pb}

From the reductions of the previous section, we can derive that the bilevel problem is at least $\cal{NP}$-hard in the strong sense. With respect to the polynomial hierarchy in complexity theory, the bilevel problem cannot be complete with respect to class $\Sigma^p_2$ of level $3$ due to the existence of the MIP formulation in  Section~\ref{subsec:mipmodel} \citep{woegingerTroubleSecondQuantifier2021}: consequently, the bilevel is $\cal{NP}$-hard in the strong sense. 

Here, we show that when $m$ is fixed, the problem is weakly $\cal{NP}$-hard by providing a moderately exponential-time dynamic programming algorithm based on the block structure of an optimal schedule as described in Section~\ref{subsec:BlockStruct}. The main idea of the dynamic programming approach is to fill the blocks one by one with the correct set of jobs. We use state variables that keep track of each possible set of completion times, the set of machines that are available, the current block, and the currently considered job. Thus, we try to assign the job to each available machine in the current block; once it is assigned, we can update the corresponding machine's completion time and proceed to the next job. When the block is full, we update the state variables for the next block and continue until all blocks are fulfilled.

% --------------------------
%  -------------------------
% 
%	 Dynamic programming	
% 
% -------------------------- 
% --------------------------

We need to select and schedule $n$ jobs from a set $\mathcal{J}$ containing $N$ jobs. Initially, we consider the case where there are no equal-size jobs, i.e., $p_j \neq p_\ell$ for all $j,\ell=1, \hdots,N$, $j\neq\ell$. We index blocks \(\mathcal{B}_b\), which represent sets of pairs (machine index, position on machine schedule), in chronological order with \( 1 \leqslant b \leqslant b_{\max} \), where block $B_{b_{\max}}$ is the final block in the schedule, which contains the final jobs on the machines with speed $V_0$.
Additionally, let $\vec{C}=\left(C_{M_i}\right)_{1 \leqslant i \leqslant m}$ denote the vector of machine completion times. To facilitate reading, we define $\omega(\vec{C},i,j)=\begin{cases}
		w_j \text{ if } C_{M_i}+p_j/V_i > d_j \\
		0 \text{ otherwise}
	\end{cases},$

and $M(\mathcal{B}_b) = \left\{i / \exists (i, k) \in \mathcal{B}_b\right\}$ as the set of machine indices in block $\mathcal{B}_b$. We also define $\vec{C'}=\vec{C} \oplus (i,p_j)$ as $ {C'}_{M_i} = C_{M_i}+p_j/V_i$ and ${C'}_{M_{\ell}} = C_{M_{\ell}}$ for $\ell \neq i$.

In our dynamic programming algorithm we work from behind. We assume that we have a schedule that follows the block structure in which each machine $M_i$ $(i=1, \ldots , m)$ has a given completion time $C_{M_i}$; the number of such possible completion times amounts to $\mathcal{O}(\sum p_j)$ per machine. In each step we select one job to fill a position on one of the machines in the current block (starting with block $b_{\max}$), until this block is full, after which we continue with the next block. To facilitate notation, we denote the set of {\em available} machines in the current step of the dynamic program by $\mathcal{A}\subset \mathcal{V}_0 \cup \mathcal{V}_1$. Here a machine is available, if it has a position in the currently considered block $\mathcal{B}_b$ to which we have not yet assigned a selected job.

We define $OPT[\vec{C},b,\mathcal{A},g]$ as the optimal value of the total weighted number of tardy jobs in the bilevel problem, given that all machine completion times are equal to $\vec{C}$, the next job will be scheduled on the current block $\mathcal{B}_b$, with the available machines $\mathcal{A}$ and where $J_g$ is the smallest indexed job that the follower can schedule. Remark here that $OPT[\vec{C},b,\mathcal{A},g]$ only accounts for the cost of the jobs that have been selected.
Since, all blocks, except maybe the first one, must be completely filled, we have that the block $\mathcal{B}_b$ with the set of available machines $\mathcal{A}$ is completely filled when $\mathrm{full}(b,\mathcal{A})$ is true with \\ $\mathrm{full}\left(b,\mathcal{A}\right) \iff \begin{cases}
		|\mathcal{A}| = n - \sum_{2 \leqslant b' \leqslant b_{\max}} |B_{b'}| & \text{if} ~ b=1 \\
		|\mathcal{A}| = 0 ~(i.e. ~ \mathcal{A} = \emptyset)                   & \text{if} ~ b>1 \\
	\end{cases}$.

We initialize the recursion with $OPT[\vec{C},b,\mathcal{A},g]=+\infty$ whenever $g>N$, and $OPT[\vec{C},b,\mathcal{A},g]=0$ if $b = b_{\max}$ and $\mathrm{full}(b,\mathcal{A})$ is true; this enforces that we have a schedule containing $n$ (yet unselected) jobs. Our recurrence relation is then

\begin{equation}
	OPT[\vec{C},b,\mathcal{A},g]= \begin{cases}
		\min\limits_{\substack{i \in \mathcal{A}                                                                                                          \\ g \leqslant j \leqslant N}}
		\left(w(\vec{C},i,j) + OPT[\vec{C} \oplus (i, p_j), b, \mathcal{A}\setminus\{i\}, j+1]\right) \text{ if $\mathrm{full}(b,\mathcal{A})$ is false } \\
		OPT[\vec{C},b+1,M(\mathcal{B}_{b+1}),g] ~ \text{ otherwise }
	\end{cases}
\end{equation}
\noindent
The first line corresponds to the case that the current block is not completely full. Hence, we select the best job $J_j$ with $g \leqslant j \leqslant N$ to be scheduled on the best available machine $M_i \in \mathcal{A}$. In the second line, the current block is completely filled, and we need to pass to the next block.

The value of the optimal solution is given by $OPT[\vec{0},1,M(\mathcal{B}_1),1]$.

Next, we focus on the general case with arbitrary processing times. We apply the formula described above and when we face $k$ equal-size jobs, we apply an algorithm that selects $\ell$ jobs, with $\ell =0, \hdots, k$. The main idea of the algorithm is that these $\ell$ jobs can completely fill the current block $\mathcal{B}_b$ and some more, and partially fill a last block. Consequently, we need to choose which machines on the last block will be selected; we enumerate all these possibilities. For each choice, we can determine the machine completion times since we have $\vec{C}$ and $\ell$ equal-size jobs with a processing time $p_j$. Then, we can find the best set $\ell$ out of these $k$ equal-size jobs by applying the algorithm that we used to prove Theorem~\ref{thm:Solve-Q|pj=p|Lex(Sum-Cj,Sum-wj)} in Section~\ref{subsubsec:FollowerPb}. Note that we can update the dynamic programming state when we use this algorithm for any value of $\ell$ and any choice of selected machines on the last block. This general case multiplies the time complexity by a factor of $2^m$ since we have to make $\mathcal{O}(2^m)$ choices and by a factor of $n \log n$ for the time complexity of the jobs' assignment.

Therefore, this dynamic programming algorithm requires $\mathcal{O} \left( 2^m \left(\sum_{j=1}^{N}p_j \right)^m \, N \, n^2 \log n \right)$ time and space. Correspondingly, the following theorem holds.

\begin{theorem}\label{thm:bilevelProblemComplexity}
	The $Q | V_i \in \{V_{0},V_{1}\},OPT-n| Lex \left( \sum_j C_j^F,\sum_j w_jU_j^L \right)$ problem is $\mathcal{NP}$-hard in the strong sense and when the number of machines $m$ is fixed it turns to be $\mathcal{NP}$-hard in the ordinary sense.
\end{theorem}

\begin{proof}
	We have shown these results for the follower's problem, which is a special case of the bilevel problem, where the leader selects $n=N$ jobs. In the scenario where the number of machines is fixed, the moderate-exponential worst-case time complexity dynamic programming runs in pseudo-polynomial time.
\end{proof}

\subsubsection{A related polynomially solvable bilevel scheduling problem}
As a side result, we consider the adversarial bilevel optimization problem, where the follower still minimizes $\sum_j C_j^F$. In this scenario, the leader takes decisions so that the optimal solution of the follower's problem is the worst possible. For the rest of the section, we assume that the jobs are sorted according the LPT rule.

\begin{theorem}
	\label{bilevelADVprobComplexity}
	The $Q | ADV-n | \sum_{j} C_j^F$ problem can be solved in $\mathcal{O}(N \log(N))$ time as follows:

	\begin{enumerate}
		\item The leader selects the $n \leqslant N$ jobs with the largest processing times,
		\item The follower applies the SPT-FAM rule, on these n jobs.
	\end{enumerate}
\end{theorem}

\begin{proof}
	The objective value of any schedule $\sigma_s$ with $n$ jobs determined by the follower can be rewritten as $\sum_{k=1}^n \omega_{[k]}p_{[k]}$, where $\omega_{[k]}$ and $p_{[k]}$ denote the $k$th smallest positional weight and the $k$th largest processing time of the selected jobs, for $k=1, \ldots , n$.

	Let us assume that the leader selects an optimal set $\mathcal{I}_1$ of $n$ jobs such that they are not the ones with the largest processing times. Suppose that the leader selects job $J_{u}$ instead of job $J_v$, whereas $p_{u} < p_v$. Let $\mathcal{I}_2$ denote the set of $n$ jobs that we obtain by replacing job $J_u$ with $J_v$ in $\mathcal{I}_1$. Let $p_{[k]}^1$ and $p_{[k]}^2$ denote the $k$th largest processing time of the jobs in the set $\mathcal{I}_1$ and $\mathcal{I}_2$, respectively. Then we see that $p_{[k]}^1 \leq p_{[k]}^2$ for all $k=1, \ldots , n$, where at least one of the inequalities is strict. Since the positional weights $w_{[k]}$ are positive and independent of the selected jobs, we see for the resulting schedules $\sigma_s^1$ and $\sigma_s^2$ created by the follower that
	\[
		\sum_{j=1}^n C_j(\sigma_s^1) -
		\sum_{j=1}^n C_j(\sigma_s^2) = \sum_{k=1}^n w_{[k]}(p_{[k]}^1-p_{[k]}^2)<0,
	\]
	which contradicts the optimality of the set $\mathcal{I}_1$.
\end{proof}

% --------------------------
%  -------------------------
% 
%	 Complexity	
% 
% -------------------------- 
% --------------------------

\subsection{Properties of the solutions}

In this section, we describe some dominance conditions between two partial solutions. Let $s$ and $s'$ be two partial solutions. We have $s$, that is composed of a partial schedule $\sigma_s$, a set of remaining jobs $\Omega_s$, and a number of jobs $k_s^1$ (respectively $k_s^0$) that must be scheduled on high-speed (respectively low-speed) machines to complete the solution, i.e. $|\sigma_s| + k_s^1 + k_s^0  = n$. Similarly, we have $s'$, that is composed of a partial schedule $\sigma_{s'}$, a set of remaining jobs $\Omega_{s'}$, and a number of jobs $k_{s'}^1$ (respectively $k_{s'}^0$) that must be scheduled on high-speed (respectively low-speed) machines to complete the solution, i.e. $|\sigma_{s'}| + k_{s'}^1 + k_{s'}^0  = n$. Schedules $\sigma_s$ and $\sigma_{s'}$ start at time zero without any unnecessary idle time. We denote by $C_{M_i}(\sigma_s)$ the completion time of machine $M_i$ in any partial solution $s$.

\begin{theorem} \label{thm:dominanceRuleMemo}
	If the following conditions hold:
	\begin{enumerate}
		\item $C_{M_i}(\sigma_s) \leqslant C_{M_i}(\sigma_{s'}) \quad \forall i \in \{1, \hdots, m\}$,
		\item $\sum\limits_{j \in \sigma_s} w_j U_j(\sigma_s) \leqslant \sum\limits_{j \in \sigma_{s'}} w_j U_j(\sigma_{s'})$,
		\item $k_s^1 \leqslant k_{s'}^1$,
		\item $k_s^0 \leqslant k_{s'}^0$,
		\item $\Omega_{s'} \subseteq \Omega_s$,
	\end{enumerate}
	then the best solution $s^*$ reachable from the partial solution $s$ is better than the best solution $s'^*$ reachable from the partial solution $s'$. We say that $s$ dominates $s'$.
\end{theorem}

\begin{proof}
	Let $(\sigma_s || \alpha_s^1 \cup \alpha_s^0 )$ be the best schedule that begins with $\sigma_s$ where $\alpha_s^1$ is a partial schedule on high-speed machines and $\alpha_s^0$ is a partial schedule on low-speed machines. Both are composed only of jobs coming from $\Omega_s$ and $|\alpha_s^1| = k_s^1$ and $|\alpha_s^0|=k_s^0$. We say that $\alpha_s^1 \cup \alpha_s^0$ is the best complementary schedule from $\Omega_{s}$ containing $k_s^1+k_s^0$ jobs.

	We denote by $\tilde{\alpha}_s^1 \cup \tilde{\alpha}_s^0$ the best complementary schedule to $\sigma_s$ with $k_s^1+k_s^0$ jobs from $\Omega_{s'}$. As we have $\Omega_{s'} \subseteq \Omega_s$, we obtain:

	\begin{equation}\nonumber
		\sum\limits_{\mathclap{j \in \left(\sigma_s || \alpha_s^1 \cup \alpha_s^0 \right)} } w_j U_j(\sigma_s || \alpha_s^1 \cup \alpha_s^0) \leqslant \sum\limits_{\mathclap{j \in \left( \sigma_s || \tilde{\alpha}_s^1 \cup \tilde{\alpha}_s^0 \right)}} w_j U_j(\sigma_s || \tilde{\alpha}_s^1 \cup \tilde{\alpha}_s^0)
	\end{equation}

	Let $\alpha_{s'}^1 \cup \alpha_{s'}^0$ be the best complementary schedule to $\sigma_{s'}$ with $k_{s'}^1 + k_{s'}^0$ jobs from $\Omega_{s'}$, and as we have $k_s^1 \leqslant k_{s'}^1$ and $k_s^0 \leqslant k_{s'}^0$, scheduling $\alpha_{s'}^1 \cup \alpha_{s'}^0$ after $\sigma_s$ guarantees that jobs in $\alpha_{s'}^1$ are executed on high-speed machines, while those in $\alpha_{s'}^0$ are executed on low-speed machines. We have:

	\begin{equation} \nonumber
		\sum\limits_{\mathclap{j \in \left(\sigma_s || \alpha_{s'}^1 \cup \alpha_{s'}^0 \right)} } w_j U_j(\sigma_s || \alpha_{s'}^1 \cup \alpha_{s'}^0) \geqslant  \sum\limits_{\mathclap{j \in \left( \sigma_s || \tilde{\alpha}_s^1 \cup \tilde{\alpha}_s^0 \right)}} w_j U_j ( \sigma_s ||\tilde{\alpha}_s^1 \cup \tilde{\alpha}_s^0) \geqslant \sum\limits_{\mathclap{j \in \left( \sigma_s || \alpha_s^1 \cup \alpha_s^0 \right)}} w_j U_j(\sigma_s || \alpha_s^1 \cup \alpha_s^0)
	\end{equation}

	Since the completion times of the machines in $s$ are smaller than those in $s'$, i.e. $\forall i \in \{1, \hdots, m\}, \, \, C_i(s) \leqslant C_i(s')$, we have:

	\begin{equation} \nonumber
		\sum\limits_{\mathclap{j \in \left(\sigma_{s'} || \alpha_{s'}^1 \cup \alpha_{s'}^0 \right)} } w_j U_j(\sigma_{s'} || \alpha_{s'}^1 \cup \alpha_{s'}^0)  \geqslant \sum\limits_{\mathclap{j \in \left( \sigma_s || \alpha_{s'}^1 \cup \alpha_{s'}^0 \right)}} w_j U_j(\sigma_s || \alpha_{s'}^1 \cup \alpha_{s'}^0) \geqslant \sum\limits_{\mathclap{j \in \left( \sigma_s || \alpha_s^1 \cup \alpha_s^0 \right)}} w_j U_j(\sigma_s || \alpha_s^1 \cup \alpha_s^0)
	\end{equation}

	So, $s$ dominates $s'$.
\end{proof}

\begin{remark}

For the leader's problem which is a selection problem, we could think about Knapsack-like dominance conditions, e.g. given two jobs $J_j$ and $J_{\ell}$, $J_j$ dominates $J_{\ell}$ if and only if $p_j \leqslant p_{\ell}$ and $d_j \geqslant d_{\ell}$ and $w_j \leqslant w_{\ell}$.
Such a condition would possibly hold in a single level optimization problem. However, in our bilevel setting, the following counterexample holds, where the leader has to select two jobs, with one machine of speed $V=1$ and three jobs defined as follow:

\begin{table}[ht!]
	\centering
	\begin{tabular}{cccc}
		\hline
		Job & $p_j$ & $d_j$ & $w_j$ \\
		\hline
		1   & 1     & 7     & 4     \\
		2   & 2     & 2     & 3     \\
		3   & 3     & 5     & 5     \\
		\hline
	\end{tabular}
\end{table}

Here, the Knapsack-like dominance condition would indicate that  $J_1$ dominates $J_3$ and, correspondingly,  the schedule ($J_2$ $J_3$) should be worse than ($J_1$ $J_2$). However, $s=(J_2 \; J_3)$ is better than $s=(J_1 \; J_2)$ with a weighted number of tardy jobs equal to 0 against 3. Indeed, the follower's block structure imposes a specific assignment for the jobs that may invalidate the above dominance condition.
\end{remark}

% --------------------------------------------
% 				EXACT ALGORITHMS
% --------------------------------------------

\section{A branch-and-bound algorithm}
\label{sec:Exact_algorithm}

%\subsection{A branch and bound algorithm}

\subsection{Branching scheme}

We propose a branch-and-bound algorithm that relies on a lower bound based on column generation. We use a branching strategy based on the block structure property of the follower's problem, i.e. we branch on the blocks in the order $\mathcal{B}_1,\mathcal{B}_2,\hdots,\mathcal{B}_{b_{\max}}$. At a given node, we have taken decisions on jobs $J_1,\hdots,J_{j-1}$, so we need to take a decision for job $J_j$. We have the following decisions:

\begin{itemize}
	\item We select job $J_j$. Then we iterate over all available locations in the first available block from the parent node: let us denote by $(i,k)$ such location. Then, we assign job $J_j$ on the machine $M_i$ at the position $k$. Thus, we create as many children nodes as available locations in this block.
	\item We do not select job $J_j$, so we remove it in the corresponding child node.
\end{itemize}

Consequently, if $N_a$ is the number of available locations in the first available block, we create exactly $(N_a+1)$ children nodes.

The nice aspect of this branching scheme is that we know the completion time of each machine. If we have $k$ equal-size jobs of which we want to select and schedule $\ell$ $(\ell=1, \ldots ,k)$ jobs, then we can use the algorithm that we used to prove Theorem~\ref{thm:Solve-Q|pj=p|Lex(Sum-Cj,Sum-wj)} to find these $\ell$ jobs. We create a node for each $\ell$ value and for each machine's selection on the last block, resulting in $\mathcal{O}(2^m n)$ nodes.

\subsection{Memorization}

\cite{shangBranchMemorizeExact2021} introduce the Branch-and-memorize paradigm, describing how to apply memorization within search tree algorithms for sequencing problems. The authors present several memorization schemes that can be used to prune the search tree and avoid solving identical or dominated subproblems.

The first scheme is solution memorization, which involves storing in memory the best solution reachable from each node. This approach can be used when the problem is decomposable, meaning that an optimal solution of a subproblem does not depend on the decisions taken to generate the subproblem. Before branching on a node, we check if there exists an optimal solution of the subproblem that has already been explored in memory. If so, we do not need to branch on this node.

The second scheme is called passive node memorization and involves storing information about the explored nodes in memory. This scheme requires a function $check(s,s')$ that returns \textit{true} if partial solution $s$ dominates partial solution $s'$. The idea is to check before branching on a node whether there exists a node in memory that dominates it.

The third and final scheme is predictive node memorization. Like passive node memorization, it shares the same idea: storing information about explored nodes in memory. Before branching on a node, we check whether there exists a node in memory that dominates it. If no such node exists, we try to create a partial solution that has not been explored in the tree and dominates the current branching node. If such a better partial solution is found, we store it in memory, and we do not branch on the current node since it is dominated.

Our problem is non-decomposable because the decision made during the partial schedule has a significant impact on the optimal solution for the remaining jobs. Moreover, each node $\mathcal{N}$ of the search tree represents a partial solution $s$, that is composed of a partial schedule $\sigma_s$, a set of remaining jobs $\Omega_s$, and a number of jobs $k_s^1$ (respectively $k_s^0$) that must be scheduled on high-speed (respectively low-speed) machines to complete the solution.

In this context, we employ passive node memorization based on the function $check(s,s')$. To that end, we associate a data structure with each node. This structure consists of: a bit set $\{0,1\}^{N}$ encoding the set of remaining jobs; an array with two parts representing the total completion time of high-speed (respectively low-speed) machines, both sorted in non-increasing order; the partial solution value $\sum\limits_{j \in \sigma_s} w_j U_j(\sigma_s)$; and integers for $k_s^1$ and $k_s^0$. Thus, given two nodes $\mathcal{N}_1$ and $\mathcal{N}_2$, their data structures enable implementation of the dominance condition from Theorem \ref{thm:dominanceRuleMemo} as a function $check(s_1,s_2)$.

We store the data structure of each explored node in a database $Db$. To accelerate the search for a dominating node, we group data structures in the database by the size $|\Omega_s|$. Thus, $Db[k]$ is the set of data structures associated with explored nodes for which $|\Omega_s|=k$. Consequently, given a node $\mathcal{N}_1$, we search for a dominating node $\mathcal{N}_2$ in $Db[k]$ (for $k=1,\dots,|\Omega_{s_1}|$).

However, the database can become full due to the exponential number of explored nodes. Therefore, we need to define a policy to clear it: at any time the database is full, we remove all nodes that have never dominated any other node.

As authored by \cite{shangBranchMemorizeExact2021}, the memorization mechanism has an impact on the choice of the most efficient search strategy. Experiments on that matter are reported in \href{sec:ExperRes}{Section 4}.

\subsection{Initial upper bound computation}

To compute a good upper bound at the root node of the search tree, we first compute a solution with the MIP by setting a time limit to 20 seconds (experimental results showed that 20 seconds is enough to often get near-optimal solutions).

We also use a heuristic that improves this solution. Let $s=(\sigma_s,\Omega_s)$ be a solution, with $\sigma_s$ the schedule, $\mathcal{I}_s$ the set of jobs that have been scheduled, and $\Omega_s$ the set of jobs that have not been selected. The main idea, is to loop over each block $\mathcal{B}_b$, for $b=1,\hdots,b_{\max}$. For a given block, we make all jobs in it available. Afterwards, we solve an assignment problem with all the jobs that can be scheduled in $\mathcal{B}_b$, i.e., jobs $J_j \in \mathcal{I}_s \cup \Omega_s$ with the processing time $p_j$ such as $\max\limits_{(i,k) \in \mathcal{B}_{b-1}}(p_{[i,k]}) \leqslant p_j \leqslant \min\limits_{(i,k) \in \mathcal{B}_{b+1}}(p_{[i,k]})$ where $p_{[i,k]}$ is the processing time of the job in position $k$ on machine $M_i$. Finally, we optimize all equal-size jobs using the same algorithm as described in the proof of Theorem~\ref{thm:Solve-Q|pj=p|Lex(Sum-Cj,Sum-wj)} in Section~\ref{subsubsec:FollowerPb}.

\subsection{Lower bound computation by Column Generation}

We propose an extension of the column generation approach presented in \cite{vandenakkerParallelMachineScheduling1999} in order to compute a lower bound at each node of the search tree.

The master problem is formulated as a set-covering problem with an exponential number of binary variables and a linear number of constraints. We use {\em machine schedules}, which are defined as a sequence of jobs that can be assigned to any machine. In the master problem, we cannot enforce the block structure, because we do not know the selected jobs. We do know, however, that the machine schedules must be part of a potential schedule constructed by the follower that follows the block structure. Hence, we can impose several constraints on the machine schedules that we consider. First of all, we know that the jobs must be in SPT order. Second, since the block structure is independent of the selection of the jobs, we know that we must fill the blocks $\mathcal{B}_b$, for $b=1, \ldots, b_{\max}$, which provides information on the number of jobs per machine. Since the blocks $\mathcal{B}_b$ are full for $b=2, \ldots, b_{\max}$, we know the resulting number of jobs for each fast and slow machine. The only freedom with respect to the number of jobs per machine schedule that we have concerns the $Q$ jobs in the first block, $\mathcal{B}_1$. If this block is not filled completely, then we must distribute these $Q$ jobs over the relevant machines. This leads to machine schedules with a minimum or maximum number of jobs, which corresponds to a difference of one job. Finally, we can limit the maximum number of jobs with equal processing time that can occur in a machine schedule. We will work this out later.

Each machine schedule $s$ is characterized by the following set of parameters:

\begin{equation}\nonumber
	a_{j,s} = \begin{cases}1 \text{ if the job $J_j$ is in the schedule $s$} \\ 0 \text{ otherwise}\end{cases}
\end{equation}

\begin{equation}\nonumber
	b_{s{\gamma}+} = \begin{cases}1 \text{ if the machine schedule $s$ is to be scheduled on a machine with speed $V_{\gamma}$} \\ \hspace{5mm} \text{and contains the maximum number of jobs} \\ 0 \text{ otherwise}\end{cases}
\end{equation}

\begin{equation}\nonumber
	b_{s{\gamma}-} = \begin{cases}1 \text{ if the machine schedule $s$ is to be scheduled on a machine with speed $V_{\gamma}$} \\ \hspace{5mm} \text{and contains the minimum number of jobs} \\ 0 \text{ otherwise}\end{cases}
\end{equation}

\begin{equation}\nonumber
	t_{j,s} = \begin{cases}1 \text{ if the job $J_j$ is late in the schedule $s$ } \\ 0 \text{ otherwise}\end{cases}
\end{equation}

These constants can be computed for any given schedule $s$.

In the column generation approach, we aim to compute a lower bound for $\sum_j w_jU_j$. Notice that we cannot compute it as in the bilevel scenario, because we cannot impose a constraint that ensures that, when $n$ jobs are selected, they must be scheduled to minimize $Lex(\sum_j C_j, \sum_j w_jU_j)$: this is a consequence of the block structure that cannot be enforced before we know the $n$ selected jobs. As a result, we focus on minimizing the single-level $\sum_j w_jU_j$ problem in a scenario where we select $n$ jobs. So, the optimal solution of our restricted master problem is a lower bound to the optimal solution of the master problem that is itself a lower bound to the optimal solution of the bilevel problem. For a schedule $s$, we associate the cost $c_{s}$ defined as follows: $c_{s} = \sum_{j=1}^N w_j t_{j,s}$.

Let $S$ be the set containing all feasible machine schedules. We introduce variables $x_{s} \in \left\{0,1\right\}$, with $s \in S$, and $x_{s} = \begin{cases}1 \text{ if the machine schedule $s$ is selected } \\ 0 \text{ otherwise}\end{cases}$

Now, our master problem consists of selecting $m$ machine schedules that contain exactly $n$ distinct jobs corresponding to the leader's decision, and must minimize the total cost. These $m$ machine schedules are partitioned in $m_0$ and $m_1$ machine schedules for the machines with speed $V_0$ and $V_1$, respectively. Furthermore, these $m_0$ machine schedules are partitioned into $m_{0+}$ and $m_{0-}$ machine schedules with a maximum and minimum number of jobs; a similar partition holds for the $m_1$ machine schedules. If $\mathcal{B}_1$ corresponds to positions on machines with speed $V_0$ (or $V_1$) only, then we can compute the corresponding values of $m_{0+}, m_{0-}, m_{1+}, m_{1-}$, and we find the following master problem formulation:

% Master PB

{\setstretch{0}
\begin{flalign}
	 & \min \; \sum_{s \in S} c_{s} \, x_{s}                     &\label{eq:objMaster}                                         \\
	 & \textbf{s.t.}                                             &\nonumber                                                    \\
	 & \sum_{s \in S} a_{j,s} \, x_{s} \leqslant 1 \hspace{5mm}   \forall j = 1, \hdots, N \label{eq:oneJobSelectedAtMostOne}  \\
	 & \sum_{s \in S} b_{s0+} x_s = m_{0+}                       &\label{eq:selectFastMaxMachines}                             \\
	 & \sum_{s \in S} b_{s0-} x_s = m_{0-}                       &\label{eq:selectFastMinMachines}                             \\
	 & \sum_{s \in S} b_{s1+} x_s = m_{1+}                       &\label{eq:selectSlowMaxMachines}                             \\
	 & \sum_{s \in S} b_{s1-} x_s = m_{1-}                       &\label{eq:selectSlowMinMachines}                             \\
	 & x_s \in \{0,1\} \hspace{5mm}                              \forall s \in S \nonumber                     
\end{flalign}
}
In this formulation Constraints (\ref{eq:oneJobSelectedAtMostOne}) ensure that each job is processed at most once. Constraints (\ref{eq:selectFastMaxMachines}) and (\ref{eq:selectFastMinMachines}) ensure that we select the correct number of machine schedules for fast machines with the maximum and minimum number of scheduled jobs. Constraints (\ref{eq:selectSlowMaxMachines}) and (\ref{eq:selectSlowMinMachines}) ensure the same for the slow machines. Constraints (\ref{eq:selectFastMaxMachines}),(\ref{eq:selectFastMinMachines}),(\ref{eq:selectSlowMaxMachines}) and (\ref{eq:selectSlowMinMachines}) ensure that we select exactly $n$ jobs, reflecting the leader's decisions.

If $\mathcal{B}_1$ is processed on both high and low speed machines, then we don't know the exact values $m_{0+}, m_{0-}, m_{1+}, m_{1-}$. Nevertheless, we know that there are $Q$ jobs to be scheduled in $\mathcal{B}_1$. Since there are $m_0$ fast machines, we know that we need a total of $m_0$ machine schedules for fast machines. Similarly, we need $m_1$ machine schedules for slow machines. Furthermore, we need a total of exactly $Q$ machines schedules with a maximum number of jobs.

Given these considerations, we can rewrite the last four constraints as follows:
\begin{itemize}
	\item Constraints (\ref{eq:selectFastMaxMachines}) and (\ref{eq:selectFastMinMachines}) become : $\sum_{s \in S} (b_{s0+} + b_{s0-}) x_s = m_{0}$. This constraint ensures that the total number of machine schedules for fast machines with minimal and maximal number of scheduled jobs selected is equal to the number of high speed machines.
	\item Constraints (\ref{eq:selectSlowMaxMachines}) and (\ref{eq:selectSlowMinMachines}) become : $\sum_{s \in S} (b_{s1+} + b_{s1-}) x_s = m_{1}$. This constraint ensures that the total number of machine schedules for slow machines with minimal and maximal number of scheduled jobs selected is equal to the number of low speed machines.
	\item We add a constraint $\sum_{s \in S} (b_{s0+}+b_{s1+}) x_s = Q$ that ensures that we select exactly the number of machines with the maximum number of scheduled jobs to be equal to the number of jobs in the first block.
\end{itemize}

We obtain the linear programming relaxation by replacing the definition set of variables $x_{s} \in \{0,1\}$ by $x_{s} \in [0,1]$.

\paragraph{Generating Columns} In the analysis below we refer to the model with constraints (\ref{eq:selectFastMaxMachines}),(\ref{eq:selectFastMinMachines}), (\ref{eq:selectSlowMaxMachines}) and (\ref{eq:selectSlowMinMachines}). The analysis can easily be adjusted to deal with the second model. We compute the reduced cost $r_{s}$ of any machine schedule $s$ as follows:

\begin{equation}
	\begin{aligned}
		r_{s} & = c_{s^{\gamma}} - \sum_{j=1}^N \lambda_j \, a_{j,s^{\gamma}} - \tilde{\lambda}_{\gamma}
	\end{aligned}
\end{equation}

where $\tilde{\lambda}_{\gamma}=\begin{cases}
		\lambda_{0+} \text{ if $s$ is a fast machine with the maximum number of jobs} \\
		\lambda_{0-} \text{ if $s$ is a fast machine with the minimum number of jobs} \\
		\lambda_{1+} \text{ if $s$ is a slow machine with the maximum number of jobs} \\
		\lambda_{1-} \text{ if $s$ is a slow machine with the minimum number of jobs}
	\end{cases}$
\\ and respectively $\lambda_{0+}$,$\lambda_{0-}$,$\lambda_{1+}$ and $\lambda_{1-}$ represent the given values of the dual variables corresponding to Constraints (\ref{eq:selectFastMaxMachines}),(\ref{eq:selectFastMinMachines}), (\ref{eq:selectSlowMaxMachines}) and (\ref{eq:selectSlowMinMachines}). Notice that $\lambda_1, \hdots, \lambda_N$ constitute the given values of the dual variables corresponding to constraints (\ref{eq:oneJobSelectedAtMostOne}).

It is well-known from the linear programming theory that a solution to a minimization problem is optimal if the reduced cost of each non basic variable is nonnegative. Therefore, we need to determine if there exists a machine schedule $s$ with such a negative cost. We solve the \textit{pricing problem}, that involves finding a machine schedule in $S$ with the minimum reduced cost. Removing the constant, we need to minimize:

\begin{equation} \nonumber
	\begin{aligned}
		r_{s} & =
		c_{s} - \sum_{j=1}^N \lambda_{j} \, a_{j,s} = \sum_{j=1}^{N} \left(w_j t_{j,s}  - \lambda_{j} \, a_{j,s}\right) \\
	\end{aligned}
\end{equation}

We achieve the calculation of the schedule with a minimal reduced cost by dynamic programming. It is applied for each speed $V_0$ and $V_1$ and for the maximum and minimum number of scheduled jobs.

Define the job groups $G_1, \dots, G_c$ such that jobs within each group have equal processing time. We assume that $p(G_1) < p(G_2) < \dots < p(G_c)$, where $p(G_g)$ denotes the processing time of any job in $G_g$. Thus, the set of jobs $\mathcal{J}$ is partitioned as $\mathcal{J} = G_1 \cup \dots \cup G_c$. Assume that we solve the dynamic programming for speed $V_\gamma$. Define $F^\gamma_k(t, g)$ as the minimum reduced cost of a machine schedule with speed $V_\gamma$ such that the first job from sets $G_g, \dots, G_c$ starts at time $t$ and where exactly $k$ jobs have been selected from the job sets $G_g, \dots, G_c$. The initialization is $F^\gamma_k(t, g) = 0$ for $k=0$, and is $+\infty$ otherwise; this implies that we must start with an empty schedule. We find the best machine schedule for a minimum and maximum number of scheduled jobs in one run of the dynamic programming, adding the multiplier $\tilde{\lambda}_\gamma$ for the corresponding constraint at the end. Given the relevant values of $F^\gamma_k(t, g)$, we compute the values of $F^\gamma_k(t, g-1)$ by taking the jobs from $G_{g-1}$ into account. If we select $\ell$ jobs from set $G_{g-1}$, then the first of these jobs will start at time $t$ and the $\ell$-th job will end at time $t + \ell p/V_\gamma$, where $p$ is the processing time of all jobs in set $G_{g-1}$. We can use a subroutine to compute the cost of completing these $\ell$ jobs, which is a polynomial-time algorithm and is described at the end of the section. We denote $S(\ell, g-1, t)$ as the cost for completing $\ell$ jobs from group $G_{g-1}$ that start at time $t$.

The advantage of recording the number of jobs already scheduled is two-fold. First, we only select machine schedules with the right number of jobs. Secondly, we can reduce the number of jobs that we select from a group $G_g$ by setting $\ell_{\max}$ as the maximum number of jobs allowed from that group. Suppose that we want to determine $F^\gamma_k(t, g)$, including the jobs from $G_g$; we need to schedule $k$ jobs on the machine with speed $V_\gamma$. Thus, using the block structure, we can identify which block $\mathcal{B}_h$ contains the $k^{\text{th}}$ job from the end on the machine. If the leader selects all jobs from $G_g$ with at least one job in block $\mathcal{B}_h$, then this will fill subsequent blocks, and we denote the last filled block as $\mathcal{B}_u$. Thus, $\ell_{\max}$ is the maximum number of jobs that can be assigned to one machine with speed $V_\gamma$ from the relevant blocks $\mathcal{B}_h$ up to and including $\mathcal{B}_u$. These computations are performed as preprocessing -- prior to solving any pricing problem -- since the block structure remains unchanged throughout the column generation process.

The recurrence relation then becomes equal to:

\[
	F^\gamma_k(t, g - 1) = \min\limits_{0\leqslant \ell \leqslant \ell_{\max}}\left\{F^\gamma_{k-\ell}(t+\ell p/V_{\gamma}, g)+S(\ell,g-1,t)\right\}
\]

In this formula $\ell$ corresponds to the number of jobs from $G_{g-1}$ that we add; here $\ell=0$ implies that no jobs from $G_{g-1}$ are added.

The optimal value $\stackrel{\star}{F^{\gamma}}$ for the machine with the speed $V_{\gamma}$ is obtained by computing $F^{\gamma}_{\eta_{\gamma+}}(0,1)$, where $\eta_{\gamma+}$ is the maximum number of jobs that can be scheduled on machine with speed $V_{\gamma}$.

% Sub problem : identical jobs

Now, let us return to the subroutine $S(\ell, g, t)$. First, if $G_g$ contains only one job denoted $J_g$, then the cost of completing this single job is given by $S(\ell, g, t)=\begin{cases}
		- \lambda_g     & \text{ if $t+p_g/V_{\gamma} \leqslant d_g$} \\
		w_g - \lambda_g & \text{ if $t+p_g/V_{\gamma} > d_g$}
	\end{cases}$

Secondly, we consider a set of equal-size jobs that start at time $t$ with $n_g$ jobs $J_1, \hdots, J_{n_g}$ in $G_g$, where $p_1=p_2=\hdots=p_{n_g}=p$. We aim to minimize the reduced cost by selecting and scheduling some of these jobs. We consider all the possible scenarios where $1$ job is selected, $2$  jobs are selected, $\hdots$, $\ell_{\max}$ jobs are selected. We solve the equivalent maximization reduced cost problem, where each coefficient is multiplied by -1. Thus, each job has a contribution defined by $(\lambda_j - w_j \, U_j$), with $U_j$ equals 1 if the job is late 0 otherwise.

Let $v_j=\lambda_j$ be the contribution of job $J_j$ when it is executed on time, and let $q_j = \lambda_j - w_j$ be its contribution of $J_j$ when it executed tardy. There are three possibilities for job $J_j$:
either it is selected as an on-time job, or as a tardy job, or not at all. A naive approach to solving this problem involves solving $\ell_{\max}$ assignment problems, one for each scenario where we select $\ell=1,\hdots,\ell_{\max}$ jobs. The total time complexity of this naive algorithm that computes all possible schedules of negative reduced cost is $\mathcal{O}({n_g}^4)$. We propose below a dedicated algorithm with a total time complexity of $\mathcal{O}({n_g}^3)$.

Since we do not know the number of jobs that need to be selected beforehand, this may force us to select a job even if its contribution is negative when tardy. To do so, we subtract some value $M$ from each job's contribution. This step aims to reduce to two the number of choices for each job. If the job is on-time or tardy but its contribution is positive, i.e. $v_j - M\geqslant0$ or $q_j-M\geqslant0$, then we select it. Otherwise, the job's contribution is negative, and we do not select it at all.  We can adjust the number of jobs that are selected by modifying the value of $M$.

Assume that $M$ is given, and define, for each job $J_j$, the value $w'_j$ as the difference between being on time with adjusted value $(v_j - M)$ and the second option where either the job is executed tardy with the adjusted value $(q_j - M)$ or it is not selected. Thus, the value $w'_j$ becomes $w'_j = (v_j - M) - \max(0, q_j - M) = \min(v_j - M, v_j - q_j)$.
The problem now essentially becomes a $1||\sum w'_jU_j$ problem, which can be solved using Dynamic Programming which assumes that the jobs are in EDD order. Let $f_j (k)$ be the total cost of the on-time jobs from $\{J_1, \ldots ,J_j\}$, where we have exactly $k$ jobs on time. The recurrence relation is
\[
	f_{j+1}(k)= \max \{f_j(k), f_j(k-1)+w'_{j+1} \} \quad \quad \mbox{if $d_{j+1} \geq t + kp$},
\]
and $f_j(k)$, otherwise. The running time of this algorithm is $\mathcal{O}({n_g}^2)$, since $k$ and $j$ run from 0 to ${n_g}$. We can reduce the running time of the above dynamic programming algorithm to
$\mathcal{O}(n_g)$ by renumbering the jobs in order of non-increasing $w'_j$ value and using a clever data structure; we describe this algorithm in \ref{sec:PricingEqualSizeJobJippeAlgo}.

We focus on choosing the value of $M$. First, we select a small perturbation $\epsilon > 0$; since scheduling or removing a job is allowed by the differences $v_j - M$ and $q_j - M$, we have at least four possible values of $M$: $M = v_j - \epsilon$, $M = v_j + \epsilon$, $M = q_j - \epsilon$, and $M = q_j + \epsilon$. Furthermore, the value of $M$ impacts $w'_j = \min(v_j - M, v_j - q_j)$; for all $M < q_j$, $w'_j$ is constant and equal to $v_j - q_j$, while it becomes linear for $M \geqslant q_j$. Since $w'_j$ determines the order in which jobs are carried out, we need to add values of $M$ where the linear part of $w'_j$ intersects the constant part of $w'_k$ for all distinct jobs $j,k=1,\hdots,n_g$, specifically at $M = v_j - v_k + q_k + \epsilon$ and $M = v_j - v_k + q_k - \epsilon$. If we must select all jobs, the value of $M$ is set to $\min\limits_{1 \leqslant j \leqslant n_g} \{q_j\} - \epsilon$.

If there are equal contributions for on-time jobs, we sort them in non-increasing order of weight $w_j$, since we prefer to select the job with the greatest weight first. Conversely, if there are equal contributions for tardy jobs, we sort them in non-decreasing order of weight $w_j$, since we prefer to select the job with the smallest weight first. After sorting, we apply a slight perturbation to the contributions to ensure that all $M$ values become distinct, leading to the computation of $\mathcal{O}(n_g^2)$ distinct values of $M$.

Finally, as the value of $M$ decreases, the number of selected jobs increases. Therefore, we sort the $M$ values in non-increasing order. If, for a given value of $M$, the number of selected jobs exceeds $\ell_{\max}$, then we can terminate the subroutine.

So, the total complexity of the algorithm that computes all possible schedules of negative reduced cost is $\mathcal{O}({n_g}^3)$ time.

Notably, in many cases, $\ell_{\max}$ is small, typically less than or equal to three. In these situations, we employ a brute-force approach that exhaustively tests all possible combinations of job selection. For the cases where $\ell_{\max} \geqslant 4$ jobs, we use the algorithm described above. For implementation details on how to optimize the column generation approach we refer to \ref{sec:OPTCGwithBaB}. Also, we attempted to use heuristics to accelerate the generation of improved columns; unfortunately, experimental results indicate that this approach does not enhance the method. Conversely, generating multiple columns during the backtracking phase in dynamic programming has shown to improve the method, as evidenced by our experiments where generating three columns during this phase yields best performances.

To resume, the pricing problem is solved by a dynamic programming that requires 

$\mathcal{O} \left( \sum_{j=1}^{N} p_j \; \times \; N \; \times \; n^3 \right)$ time and space in the worst case.

\section{Experimental results}
\label{sec:ExperRes}

In this section we computationally compared CPLEX solver version 22.1.1. applied to the MIP formulation (denoted as MIP) and the branch-and-bound algorithm (denoted as BaB). All algorithms were coded in langague C++, and the tests were run on a PC 2GHz AMD Epyc 7702. The source code can be found on the GitHub repository at \\ \href{https://github.com/QuentinSchau/Exact\_Algorithms\_Bilevel\_Scheduling\_Industry\_Futur}{https://github.com/QuentinSchau/Exact\_Algorithms\_Bilevel\_Scheduling\_Industry\_Futur}. We randomly generated instances where we have $m_1$ high-speed machines with their speed $V_1=2$ and $m_0$ low-speed machines with their speed $V_0=1$. The processing times were drawn at random from the uniform distribution $U[1,100]$ and weights from the uniform distribution $U[1,10]$. Due dates were generated using the uniform distribution $U[P \times (1-tf-rdd/2),P \times (1-tf+rdd/2)]$ with $tf,rdd \in \{0.2,0.4,0.6,0.8,1.0\}$ and $P=\frac{\sum_{j=1}^N p_j}{m_1 \times V_1 + m_0 \times V_0}$. Here, $tf$ and $rdd$ define a class of instances and we have 25 classes of instances. For each class, we randomly generated 10 instances of size $N=40,50,60,70,80$. For each value of $N$ we set $n=N/4,N/2,3N/4$, where values were rounded down in case they were fractional. Moreover, we set the total number of machines to $m=2,4$, where we have the same number of high-speed and low-speed machines.

Primary experimental results indicated us that using memorization yields the best results for the BaB algorithm when a depth-first search strategy is used. In the remainder, we consider this version of BaB.

We divide these instances into "easy" instances for which, all exact methods solve them to optimality in at most 300 seconds and "hard" instances for which, at least one method does not solve them to optimality within this time budget.

Table \ref{tab:TimeAllOpt} presents the results on easy instances. In this table, we report the number of instances solved to optimality (\#OPT) out of the 250 per class by both algorithms. For each algorithm, we provide the minimum, average, and maximum computation times, in seconds, for each value of $N$ and $n$.

Table \ref{tab:TimeOther} presents the results for "hard" instances. There, \#OPT shows the number of instances solved to optimality by each method, as well as the minimum, average, and maximum computation times.

Tables \ref{tab:NbNodesAllMethod} and \ref{tab:NbNodes} present the minimum, average, and maximum number of nodes explored by both methods for easy and "hard" instances, respectively.

Tables \ref{tab:TimeAllOpt} and \ref{tab:TimeOther} demonstrate that BaB enables solving more instances to optimality with a significant improvement in computation time compared to MIP. MIP can solve to optimality all instances with $n=N/2$ and $N=40$ on two and four machines, whereas BaB can solve all instances for $n=N/2$ until $N=60$ for two machines and until $N=50$ for four machines. Additionally, for $n=3N/4$ and for two machines, MIP does not solve to optimality all instances for $N=40$, whereas BaB can solve these instances for $N=50$. Furthermore, although MIP can solve more instances for $n=3N/4$ and $N=80$ for four machines on hard instances, on average on easy instances, BaB is faster. Nevertheless, in terms of computation time, BaB is generally faster than MIP.

Moreover, Tables \ref{tab:NbNodesAllMethod} and \ref{tab:NbNodes} demonstrate that for instances solved optimally with $n=N/4$ and $N=40,50,60$ on two machines and for $N=40,50,60,70$ on four machines, the number of nodes explored by BaB is 0. Since we run MIP for several seconds to compute an upper bound, it follows that MIP has identified and validated the optimal solution.
For hard instances, BaB explores significantly fewer nodes (at least a factor of 2 reduction on average until $N=80$ and $n=N/2$) compared to MIP for two machines. However, for instances with four machines, MIP appears to explore fewer nodes. This disparity can be attributed to the branching schemes used by each algorithm, specifically $\mathcal{O}((m+1)^N)$ for BaB . Although this scheme remains competitive with MIP for two machines, its exponential growth with respect to the number of machines has a negative impact on BaB's performance due to the increased number of lower bound computations, which are relatively time-consuming.

Since the leader must select $n$ out of $N$ jobs, leading to ${N}\choose{n}$ possibilities, the combinatorics of selection is greatest when $n=N/2$. Furthermore, we observe that this value does not necessarily lead to the most challenging scenario for the bilevel problem, whereas it does for the leader's problem. This difference can be attributed to the block structure of the follower's problem, as having more jobs allows for more permutations within this block compared to the selection options of the leader.

\begin{landscape}
	\setlength{\tabcolsep}{3pt}
	\begin{table}[ht!]
		\begin{minipage}{.5\linewidth}
			\centering
			\begin{tabular}{ccllllllll}
				\toprule
				                      &                       &     &         & \multicolumn{3}{c}{MIP} & \multicolumn{3}{c}{BaB}                                                                   \\
				$m$                   & $N$                   & $n$ & \#OPT   & $t_{min}$               & $t_{avg}$               & $t_{max}$       & $t_{min}$ & $t_{avg}$       & $t_{max}$       \\
				\midrule
				\multirow[]{15}{*}{2} & \multirow[]{3}{*}{40} & 10  & 250/250 & 0.05                    & 0.08                    & 0.33            & 0.05      & 0.09            & 0.32            \\
				\cline{3-10}
				                      &                       & 20  & 250/250 & 0.16                    & 4.41                    & 238.83          & 0.16      & \textbf{1.63  } & \textbf{23.17 } \\
				\cline{3-10}
				                      &                       & 30  & 222/250 & 0.38                    & 17.08                   & 281.96          & 0.38      & \textbf{8.55  } & \textbf{45.97 } \\
				\cline{2-10} \cline{3-10}
				                      & \multirow[]{3}{*}{50} & 12  & 250/250 & 0.08                    & 0.14                    & 3.18            & 0.08      & 0.15            & 3.13            \\
				\cline{3-10}
				                      &                       & 25  & 240/250 & 0.30                    & 10.02                   & 289.95          & 0.29      & \textbf{3.91  } & \textbf{27.90 } \\
				\cline{3-10}
				                      &                       & 37  & 157/250 & 0.76                    & 48.83                   & 271.72          & 0.76      & \textbf{28.60 } & \textbf{247.08} \\
				\cline{2-10} \cline{3-10}
				                      & \multirow[]{3}{*}{60} & 15  & 250/250 & 0.15                    & 0.30                    & 3.23            & 0.15      & 0.30            & 3.23            \\
				\cline{3-10}
				                      &                       & 30  & 219/250 & 0.49                    & 15.57                   & 208.50          & 0.51      & \textbf{7.90  } & \textbf{56.44 } \\
				\cline{3-10}
				                      &                       & 45  & 75/250  & 1.26                    & 37.04                   & 279.53          & 1.27      & \textbf{27.22 } & \textbf{234.12} \\
				\cline{2-10} \cline{3-10}
				                      & \multirow[]{3}{*}{70} & 17  & 250/250 & 0.20                    & 0.72                    & 43.22           & 0.20      & \textbf{0.63  } & \textbf{20.46 } \\
				\cline{3-10}
				                      &                       & 35  & 166/250 & 0.76                    & 26.36                   & 284.26          & 0.78      & \textbf{14.30 } & \textbf{250.45} \\
				\cline{3-10}
				                      &                       & 52  & 55/250  & 1.82                    & 25.96                   & \textbf{225.60} & 1.84      & \textbf{21.44 } & 278.99          \\
				\cline{2-10} \cline{3-10}
				                      & \multirow[]{3}{*}{80} & 20  & 249/250 & 0.25                    & 4.05                    & 280.34          & 0.24      & \textbf{1.81  } & \textbf{24.32 } \\
				\cline{3-10}
				                      &                       & 40  & 136/250 & 1.05                    & 15.06                   & 235.28          & 1.08      & \textbf{9.08  } & \textbf{172.46} \\
				\cline{3-10}
				                      &                       & 60  & 44/250  & 2.46                    & 12.08                   & 256.70          & 2.47      & \textbf{8.23  } & \textbf{96.26 } \\
				\cline{2-10} \cline{3-10} \bottomrule
				\multirow[]{15}{*}{4} & \multirow[]{3}{*}{40} & 10  & 250/250 & 0.06                    & 0.10                    & 0.20            & 0.06      & 0.10            & 0.22            \\
				\cline{3-10}
				                      &                       & 20  & 250/250 & 0.17                    & 2.63                    & 69.37           & 0.17      & \textbf{2.36  } & \textbf{45.78 } \\
				\cline{3-10}
				                      &                       & 30  & 132/250 & 0.41                    & 40.36                   & 292.22          & 0.41      & \textbf{33.97 } & \textbf{285.52} \\
				\cline{2-10} \cline{3-10}
				                      & \multirow[]{3}{*}{50} & 12  & 250/250 & 0.08                    & 0.14                    & 0.44            & 0.08      & 0.14            & 0.43            \\
				\cline{3-10}
				                      &                       & 25  & 232/250 & 0.35                    & 29.90                   & 294.48          & 0.35      & \textbf{15.98 } & \textbf{293.98} \\
				\cline{3-10}
				                      &                       & 37  & 78/250  & 0.74                    & 23.06                   & 286.44          & 0.73      & \textbf{13.44 } & \textbf{171.93} \\
				\cline{2-10} \cline{3-10}
				                      & \multirow[]{3}{*}{60} & 15  & 250/250 & 0.19                    & 0.39                    & 4.18            & 0.20      & 0.39            & 4.17            \\
				\cline{3-10}
				                      &                       & 30  & 169/250 & 0.46                    & 20.27                   & 264.98          & 0.47      & \textbf{10.48 } & \textbf{157.46} \\
				\cline{3-10}
				                      &                       & 45  & 66/250  & 1.22                    & 10.87                   & 155.13          & 1.28      & \textbf{5.05  } & \textbf{30.88 } \\
				\cline{2-10} \cline{3-10}
				                      & \multirow[]{3}{*}{70} & 17  & 250/250 & 0.15                    & 0.84                    & 17.11           & 0.15      & 0.85            & 17.07           \\
				\cline{3-10}
				                      &                       & 35  & 144/250 & 0.67                    & 9.37                    & 292.39          & 0.67      & \textbf{8.02  } & \textbf{285.59} \\
				\cline{3-10}
				                      &                       & 52  & 55/250  & 1.99                    & 7.38                    & \textbf{137.67} & 1.99      & \textbf{6.31  } & 140.04          \\
				\cline{2-10} \cline{3-10}
				                      & \multirow[]{3}{*}{80} & 20  & 250/250 & 0.24                    & 1.70                    & 45.36           & 0.24      & \textbf{1.52  } & \textbf{23.20 } \\
				\cline{3-10}
				                      &                       & 40  & 132/250 & 1.01                    & \textbf{5.43  }         & \textbf{106.35} & 1.00      & 7.28            & 274.19          \\
				\cline{3-10}
				                      &                       & 60  & 53/250  & 2.71                    & 9.39                    & 153.82          & 2.71      & \textbf{5.28  } & \textbf{20.43 } \\
				\cline{2-10} \cline{3-10}
				\bottomrule
			\end{tabular}
			\caption{Computation times on easy instances}
			\label{tab:TimeAllOpt}
		\end{minipage}%
		\begin{minipage}{.5\linewidth}
			\centering
			\begin{tabular}{cclllllllll}
				\toprule
				                      &                       &     & \multicolumn{4}{c}{MIP} & \multicolumn{4}{c}{BaB}                                                                                \\
				$m$                   & $N$                   & $n$ & \#OPT                   & $t_{min}$               & $t_{avg}$ & $t_{max}$ & \#OPT            & $t_{min}$ & $t_{avg}$ & $t_{max}$ \\
				\midrule
				\multirow[]{10}{*}{2} & 40                    & 30  & 0/28                    & 300.00                  & 300.00    & 300.01    & \textbf{28/28 }  & 21.01     & 28.69     & 39.76     \\
				\cline{2-11} \cline{3-11}
				                      & \multirow[]{2}{*}{50} & 25  & 0/10                    & 300.00                  & 300.00    & 300.01    & \textbf{10/10 }  & 20.76     & 22.85     & 25.01     \\
				\cline{3-11}
				                      &                       & 37  & 0/93                    & 300.00                  & 300.01    & 300.01    & \textbf{93/93 }  & 29.29     & 60.14     & 185.26    \\
				\cline{2-11} \cline{3-11}
				                      & \multirow[]{2}{*}{60} & 30  & 0/31                    & 300.00                  & 300.01    & 300.01    & \textbf{31/31 }  & 22.08     & 33.58     & 72.06     \\
				\cline{3-11}
				                      &                       & 45  & 1/175                   & 90.91                   & 298.81    & 300.01    & \textbf{155/175} & 28.31     & 153.69    & 300.58    \\
				\cline{2-11} \cline{3-11}
				                      & \multirow[]{2}{*}{70} & 35  & 0/84                    & 300.01                  & 300.01    & 300.01    & \textbf{83/84 }  & 35.30     & 77.21     & 300.04    \\
				\cline{3-11}
				                      &                       & 52  & 11/195                  & 58.21                   & 290.34    & 300.02    & \textbf{59/195}  & 73.46     & 268.91    & 301.30    \\
				\cline{2-11} \cline{3-11}
				                      & \multirow[]{3}{*}{80} & 20  & 0/1                     & 300.01                  & 300.01    & 300.01    & \textbf{1/1   }  & 21.31     & 21.31     & 21.31     \\
				\cline{3-11}
				                      &                       & 40  & 1/114                   & 261.60                  & 299.67    & 300.02    & \textbf{107/114} & 33.22     & 117.99    & 302.06    \\
				\cline{3-11}
				                      &                       & 60  & 9/206                   & 26.99                   & 294.52    & 300.02    & \textbf{16/206}  & 58.59     & 293.56    & 301.63    \\
				\cline{2-11} \cline{3-11}\bottomrule
				\multirow[]{9}{*}{4}  & 40                    & 30  & 19/118                  & 31.58                   & 274.93    & 300.01    & \textbf{48/118}  & 22.40     & 225.61    & 300.02    \\
				\cline{2-11} \cline{3-11}
				                      & \multirow[]{2}{*}{50} & 25  & 0/18                    & 300.00                  & 300.01    & 300.01    & \textbf{18/18 }  & 21.40     & 49.84     & 144.24    \\
				\cline{3-11}
				                      &                       & 37  & 8/172                   & 38.49                   & 294.23    & 300.01    & \textbf{18/172}  & 31.21     & 284.66    & 300.18    \\
				\cline{2-11} \cline{3-11}
				                      & \multirow[]{2}{*}{60} & 30  & 2/81                    & 122.81                  & 296.64    & 300.01    & \textbf{61/81 }  & 21.15     & 149.51    & 300.04    \\
				\cline{3-11}
				                      &                       & 45  & 6/184                   & 34.65                   & 293.26    & 300.02    & \textbf{8/184 }  & 31.36     & 291.16    & 300.29    \\
				\cline{2-11} \cline{3-11}
				                      & \multirow[]{2}{*}{70} & 35  & 3/106                   & 52.62                   & 294.08    & 300.02    & \textbf{39/106}  & 24.85     & 240.49    & 300.48    \\
				\cline{3-11}
				                      &                       & 52  & 4/195                   & 62.19                   & 296.77    & 300.02    & \textbf{7/195 }  & 118.10    & 295.80    & 300.86    \\
				\cline{2-11} \cline{3-11}
				                      & \multirow[]{2}{*}{80} & 40  & 6/118                   & 18.87                   & 288.99    & 300.02    & \textbf{15/118}  & 36.42     & 276.95    & 300.49    \\
				\cline{3-11}
				                      &                       & 60  & \textbf{11/197}         & 30.08                   & 290.52    & 300.02    & 0/197            & 300.00    & 300.03    & 301.21    \\
				\cline{2-11} \cline{3-11}
				\bottomrule
			\end{tabular}
			\caption{Computation times on hard instances}
			\label{tab:TimeOther}
		\end{minipage}%
	\end{table}

	\begin{table}[ht!]
		\setlength{\tabcolsep}{2pt}
		\begin{minipage}{.5\linewidth}
			\centering
			\begin{tabular}{cclllllll}
				\toprule
				                      &                       &     & \multicolumn{3}{c}{MIP} & \multicolumn{3}{c}{BaB}                                                                                                                                                     \\
				$m$                   & $N$                   & $n$ & $\mathcal{N}_{min}$     & $\mathcal{N}_{avg}$                 & $\mathcal{N}_{max}$                 & $\mathcal{N}_{min}$ & $\mathcal{N}_{avg}$                 & $\mathcal{N}_{max}$                 \\
				\midrule
				\multirow[]{15}{*}{2} & \multirow[]{3}{*}{40} & 10  & 0                       & 11.12                               & 1223                                & 0                   & \textbf{0        }\textbf{        } & \textbf{0        }\textbf{        } \\
				\cline{3-9}
				                      &                       & 20  & 0                       & 1750.36                             & 97543                               & 0                   & \textbf{14.42    }\textbf{        } & \textbf{1300     }\textbf{        } \\
				\cline{3-9}
				                      &                       & 30  & 0                       & 4743.21                             & 125273                              & 0                   & \textbf{280.86   }\textbf{        } & \textbf{5593     }\textbf{        } \\
				\cline{2-9} \cline{3-9}
				                      & \multirow[]{3}{*}{50} & 12  & 0                       & 2.19                                & 158                                 & 0                   & \textbf{0        }\textbf{        } & \textbf{0        }\textbf{        } \\
				\cline{3-9}
				                      &                       & 25  & 0                       & 2815.82                             & 59516                               & 0                   & \textbf{44.50    }\textbf{        } & \textbf{1510     }\textbf{        } \\
				\cline{3-9}
				                      &                       & 37  & 0                       & 9595.18                             & 94478                               & 0                   & \textbf{1356.18  }\textbf{        } & \textbf{17518    }\textbf{        } \\
				\cline{2-9} \cline{3-9}
				                      & \multirow[]{3}{*}{60} & 15  & 0                       & 73.30                               & 2637                                & 0                   & \textbf{0        }\textbf{        } & \textbf{0        }\textbf{        } \\
				\cline{3-9}
				                      &                       & 30  & 0                       & 3955.66                             & 62184                               & 0                   & \textbf{218.45   }\textbf{        } & \textbf{6404     }\textbf{        } \\
				\cline{3-9}
				                      &                       & 45  & 0                       & 3319.24                             & \textbf{20040    }\textbf{        } & 0                   & \textbf{909.21   }\textbf{        } & 20398                               \\
				\cline{2-9} \cline{3-9}
				                      & \multirow[]{3}{*}{70} & 17  & 0                       & 255.98                              & 23073                               & 0                   & \textbf{0.09     }\textbf{        } & \textbf{22       }\textbf{        } \\
				\cline{3-9}
				                      &                       & 35  & 0                       & 3949.63                             & 47013                               & 0                   & \textbf{556.61   }\textbf{        } & \textbf{31449    }\textbf{        } \\
				\cline{3-9}
				                      &                       & 52  & 0                       & 2019.62                             & 15805                               & 0                   & \textbf{792.40   }\textbf{        } & \textbf{12706    }\textbf{        } \\
				\cline{2-9} \cline{3-9}
				                      & \multirow[]{3}{*}{80} & 20  & 0                       & 1257.62                             & 87631                               & 0                   & \textbf{9.67     }\textbf{        } & \textbf{442      }\textbf{        } \\
				\cline{3-9}
				                      &                       & 40  & 0                       & 1696.01                             & 27584                               & 0                   & \textbf{337.74   }\textbf{        } & \textbf{15842    }\textbf{        } \\
				\cline{3-9}
				                      &                       & 60  & 0                       & 301.95                              & 6455                                & 0                   & \textbf{178.77   }\textbf{        } & \textbf{4303     }\textbf{        } \\
				\cline{2-9} \cline{3-9}\bottomrule
				\multirow[]{15}{*}{4} & \multirow[]{3}{*}{40} & 10  & 0                       & 11.53                               & 651                                 & 0                   & \textbf{0        }\textbf{        } & \textbf{0        }\textbf{        } \\
				\cline{3-9}
				                      &                       & 20  & 0                       & 1678.14                             & 29660                               & 0                   & \textbf{229.43   }\textbf{        } & \textbf{47437    }\textbf{        } \\
				\cline{3-9}
				                      &                       & 30  & 0                       & \textbf{12026.83 }\textbf{        } & \textbf{142538   }\textbf{        } & 0                   & 14963.45                            & 170468                              \\
				\cline{2-9} \cline{3-9}
				                      & \multirow[]{3}{*}{50} & 12  & 0                       & 25.62                               & 1939                                & 0                   & \textbf{0        }\textbf{        } & \textbf{0        }\textbf{        } \\
				\cline{3-9}
				                      &                       & 25  & 0                       & 10522.53                            & \textbf{115001   }\textbf{        } & 0                   & \textbf{7275.95  }\textbf{        } & 178482                              \\
				\cline{3-9}
				                      &                       & 37  & 0                       & 3196.86                             & 45725                               & 0                   & \textbf{1060.41  }\textbf{        } & \textbf{25233    }\textbf{        } \\
				\cline{2-9} \cline{3-9}
				                      & \multirow[]{3}{*}{60} & 15  & 0                       & 120.92                              & 3581                                & 0                   & \textbf{0        }\textbf{        } & \textbf{0        }\textbf{        } \\
				\cline{3-9}
				                      &                       & 30  & 0                       & 4412.47                             & \textbf{59521    }\textbf{        } & 0                   & \textbf{3662.15  }\textbf{        } & 104169                              \\
				\cline{3-9}
				                      &                       & 45  & 0                       & 910.39                              & 11873                               & 0                   & \textbf{30.05    }\textbf{        } & \textbf{770      }\textbf{        } \\
				\cline{2-9} \cline{3-9}
				                      & \multirow[]{3}{*}{70} & 17  & 0                       & 486.33                              & 15517                               & 0                   & \textbf{0        }\textbf{        } & \textbf{0        }\textbf{        } \\
				\cline{3-9}
				                      &                       & 35  & 0                       & \textbf{1193.62  }\textbf{        } & \textbf{53012    }\textbf{        } & 0                   & 2880.17                             & 201848                              \\
				\cline{3-9}
				                      &                       & 52  & 0                       & \textbf{270.25   }\textbf{        } & \textbf{5441     }\textbf{        } & 0                   & 2923.82                             & 160809                              \\
				\cline{2-9} \cline{3-9}
				                      & \multirow[]{3}{*}{80} & 20  & 0                       & 765.20                              & 29961                               & 0                   & \textbf{27.23    }\textbf{        } & \textbf{2811     }\textbf{        } \\
				\cline{3-9}
				                      &                       & 40  & 0                       & \textbf{550.84   }\textbf{        } & \textbf{12639    }\textbf{        } & 0                   & 3044.73                             & 221898                              \\
				\cline{3-9}
				                      &                       & 60  & 0                       & 260.09                              & 10272                               & 0                   & \textbf{0.04     }\textbf{        } & \textbf{1        }\textbf{        } \\
				\cline{2-9} \cline{3-9}
				\bottomrule
			\end{tabular}
			\caption{Number of explored nodes on easy instances}
			\label{tab:NbNodesAllMethod}
		\end{minipage}%
		\begin{minipage}{.5\linewidth}
			\centering
			\begin{tabular}{cclllllll}
				\toprule
				                      &                       &     & \multicolumn{3}{c}{MIP} & \multicolumn{3}{c}{BaB}                                                                                         \\
				$m$                   & $N$                   & $n$ & $\mathcal{N}_{min}$     & $\mathcal{N}_{avg}$     & $\mathcal{N}_{max}$ & $\mathcal{N}_{min}$ & $\mathcal{N}_{avg}$ & $\mathcal{N}_{max}$ \\
				\midrule
				\multirow[]{10}{*}{2} & 40                    & 30  & 13370                   & 32334.43                & 107382              & 79                  & \textbf{1566.07 }   & \textbf{5127    }   \\
				\cline{2-9} \cline{3-9}
				                      & \multirow[]{2}{*}{50} & 25  & 31952                   & 63335.10                & 120191              & 84                  & \textbf{498.80  }   & \textbf{1294    }   \\
				\cline{3-9}
				                      &                       & 37  & 3510                    & 28523.70                & 89151               & 192                 & \textbf{3338.08 }   & \textbf{18599   }   \\
				\cline{2-9} \cline{3-9}
				                      & \multirow[]{2}{*}{60} & 30  & 14447                   & 39547.10                & 151066              & 126                 & \textbf{1260.06 }   & \textbf{5371    }   \\
				\cline{3-9}
				                      &                       & 45  & 1062                    & 17580.82                & 65564               & 150                 & \textbf{5918.02 }   & \textbf{36704   }   \\
				\cline{2-9} \cline{3-9}
				                      & \multirow[]{2}{*}{70} & 35  & 5318                    & 31884.17                & 115348              & 254                 & \textbf{2313.39 }   & \textbf{14788   }   \\
				\cline{3-9}
				                      &                       & 52  & 905                     & 11075.91                & \textbf{31534   }   & 374                 & \textbf{6791.83 }   & 33511               \\
				\cline{2-9} \cline{3-9}
				                      & \multirow[]{3}{*}{80} & 20  & 87389                   & 87389                   & 87389               & 44                  & \textbf{44      }   & \textbf{44      }   \\
				\cline{3-9}
				                      &                       & 40  & 4078                    & 22124.56                & 46693               & 204                 & \textbf{3810.09 }   & \textbf{23907   }   \\
				\cline{3-9}
				                      &                       & 60  & 651                     & 8956.86                 & \textbf{23618   }   & 17                  & \textbf{7103.34 }   & 30108               \\
				\cline{2-9} \cline{3-9} \bottomrule
				\multirow[]{9}{*}{4}  & 40                    & 30  & 7650                    & \textbf{51219.26}       & \textbf{119710  }   & 2171                & 134821.12           & 260359              \\
				\cline{2-9} \cline{3-9}
				                      & \multirow[]{2}{*}{50} & 25  & 33910                   & 61079.11                & \textbf{104747  }   & 505                 & \textbf{23021.56}   & 120436              \\
				\cline{3-9}
				                      &                       & 37  & 5323                    & \textbf{27257.72}       & \textbf{56562   }   & 1683                & 119503.60           & 287996              \\
				\cline{2-9} \cline{3-9}
				                      & \multirow[]{2}{*}{60} & 30  & 20591                   & \textbf{46902.17}       & \textbf{108168  }   & 270                 & 78885.80            & 263596              \\
				\cline{3-9}
				                      &                       & 45  & 1428                    & \textbf{11963.22}       & \textbf{34792   }   & 716                 & 105314.18           & 252623              \\
				\cline{2-9} \cline{3-9}
				                      & \multirow[]{2}{*}{70} & 35  & 10431                   & \textbf{25746.80}       & \textbf{62082   }   & 462                 & 97975.81            & 232447              \\
				\cline{3-9}
				                      &                       & 52  & 668                     & \textbf{8647.37 }       & \textbf{27620   }   & 2688                & 98116.84            & 281695              \\
				\cline{2-9} \cline{3-9}
				                      & \multirow[]{2}{*}{80} & 40  & 1951                    & \textbf{18563.06}       & \textbf{64581   }   & 1523                & 94636.77            & 309197              \\
				\cline{3-9}
				                      &                       & 60  & 532                     & \textbf{6536.15 }       & \textbf{21029   }   & 7254                & 106843.60           & 294649              \\
				\cline{2-9} \cline{3-9}
				\bottomrule
			\end{tabular}
			\caption{Number of explored nodes on hard instances}
			\label{tab:NbNodes}
		\end{minipage}%
	\end{table}
\end{landscape}

\section{Conclusion}
\label{sec:CCL}

In this paper, we tackled a challenging bilevel parallel machine scheduling problem in an optimistic setting, i.e. when the follower returns among all its optimal solutions, the one that leads to the leader's optimal value. We have shown the $\mathcal{NP}$-hardness in a strong sense of the bilevel problem. We also developed several exact approaches, namely, a moderately exponential-time dynamic programming algorithm, a MIP formulation and a branch-and-bound algorithm using memorization mechanism in which we use column generation to compute a lower bound.

While the moderately exponential-time dynamic programming algorithm is only relevant for theoretical purposes, our experimental results demonstrate that the branch-and-bound algorithm outperforms the MIP on all instances, except for $N=80$ and $n=3N/4$ on four machines, but its performance deteriorates when dealing with more machines. Both approaches can efficiently solve all instances with up to 40 jobs and 2 or 4 machines.

These results show that the considered bilevel scheduling problem is very challenging and remains difficult to solve in practice, even for a dedicated exact algorithm. At first glance, this seems surprising as optimizing only the follower's objective function can be done in polynomial time. Moreover, the hardest scenario for the leader's problem is not the hardest for the bilevel problem. As a future research direction, we intend to develop heuristic approaches to tackle large instances. These heuristics could exploit the BaB algorithm as a beam-search heuristic or leverage machine learning techniques to approximate the follower's decisions.
\appendix

\section{Optimization of the pricing algorithm}\label{sec:PricingEqualSizeJobJippeAlgo}

In this section, we describe an algorithm with improved computational complexity compared to the dynamic programming approach for computing all possible schedules with negative reduced cost in the case of $n_g$ equal-sized jobs. Recall that this problem boils down to $1|p_j=p|\sum w_j'U_j$, where $w_j'=\min \{v_j-M, q_j-M\}$. To ease notation, we assume without loss of generality that $p=1$ and that we start at time 0 on the machine instead of starting at time $t$. To compensate for this, we put the due dates equal to
\[
	d_j \leftarrow \lfloor \frac{d_j-t'}{p}\rfloor;
\]
\noindent
rounding down the due dates does not change the problem, since the jobs must finish at fixed times $t$, which are limited to values within the set $\left\{1,2,\ldots,n_g \right\}$. For the remainder of this section, whenever we refer to $t$, it will represent any possible value from this predetermined set. We renumber the jobs in order of non-increasing $w'_j$ value; initially the set of on-time jobs is empty. We add the jobs one by one in order of index to the current set of on time jobs. If the job can be on time together with the already selected on-time jobs, then we put it on time. Otherwise, we mark the job as tardy and execute it at the end of the schedule if $q_j - M > 0$, and discard it if $q_j \leq M$. To test whether the next job $J_j$ can be set as on time, we need some an appropriate data structure. Any job with a due date smaller than or equal to zero will be marked as tardy immediately and not be considered as a candidate for the on time set.

We first focus on an efficient way to decide whether or not a job $J_j$ can be scheduled on time together with the already selected on-time jobs. For a given set of on-time jobs, we keep track of the number of jobs $n_t$ in the on time set that have a due date that is smaller than or equal to $t$, for all $t=1, \ldots , n_g$. It is easily verified that the jobs in the on time set can only be scheduled on time together if $n_t \leq t$, for all $t$. If we add job $J_j$ to the on-time set, then this implies that $n_t$ will increase by $1$ for all $t$ such as $t \geq d_j$ and remain equal for all $t$ with $t <d_j$. If all new $n_t$ values remain smaller than or equal to $t$, then job $J_j$ is added to the on-time set. Updating this in a straightforward manner takes $O({n_g})$ time, and leads to an $O({n_g}^2)$ algorithm. To speed up the process of updating the $n_t$ values, we compute the values $k_t=t-n_t$. The feasibility of the set of on-time jobs is guaranteed as long as $k_t \geqslant 0,\forall t$.

Determine for each $t=1, 2,\ldots , {n_g}$ the value $m_t$ as the minimum $k_{t^{\star}}$ value for all $t^{\star} \geq t$, and $t^{\star} \in \left\{1,2,\ldots,n_g \right\}$ that is, $m_t = \min\limits_{t^{\star} \geq t} k_{t^{\star}}$. If we plot the $m_t$ values as a function of $t$, then we would see a number of horizontal line segments that move up. We are only interested in the jump points $(m_t,t)$. Therefore, we store the $t$ values corresponding to a jump point and add a pointer to the predecessor and successor $t$ value for each jump point.

When we consider a new job $J_j$ with due date $d_j$, we first determine the largest value $t_1$ such that $t_1 \leqslant d_j < t_2$, where $t_1$ and $t_2$ correspond to jump points. Let $t_{0}$ be the value of the jump point that precedes $t_1$. Adding job $J_j$ to the set of on time jobs implies that $m_t$ is decreased by 1, $\forall t \geqslant t_1$, which is feasible if they all remain nonnegative. To avoid useless updates, we rather store the difference between two consecutive jump points, i.e. $(m_{t_{k+1}}-m_{t_k})$. Consequently, adding job $J_j$ only requires to decrease by 1 the value $(m_{t_1} - m _{t_0})$. If this difference becomes 0, then this means that $t_1$ does not correspond to a jump point anymore. This is easily taken care of by adjusting the pointers: the successor of $t_{0}$ becomes $t_2$ and the predecessor of $t_2$ becomes $t_{0}$.
We further mark job $J_j$ as on-time and add it to the on time set. If $t_1$ is the origin point, i.e., $t_{0}$ does not exist then it means that the job cannot be on time because $m_{t_1}=0$. At the end of the algorithm, we then process the jobs in the on-time set by increasing order of their due date.

The remaining point is how to efficiently compute the value $t_1$ for each job $J_j$. This can be done using
a disjoint set union find. This data structure supports taking the union of two segments and finding the segment to which a certain value belongs. Its running time is $\mathcal{O}(a(n))$, where $a(n)$ is the inverse of the Ackermann function \citep{tarjanWorstcaseAnalysisSet1984}. This is less than 5 for reasonable $n$ values.

We now focus on how to efficiently update the $w_j'$ when the value $M$ changes. We initially create two lists of the jobset: one has the jobs sorted by non-increasing order of the $v_j$'s while, in the other, the jobs are sorted by non-increasing order of the $(v_j - q_j)$'s. As $w_j'=\min(v_j-M,v_j-q_j)$ we compute them by a simple pass through the two lists. Notice that, due to the sorting of the two lists, the $w_j'$'s are computed in non-increasing order of their computed value. Consequently, the running time of computing and sorting the $w'_j$ decreases from $\mathcal{O}({n_g}\log({n_g}))$ to $\mathcal{O}({n_g})$.

So, the time complexity of this algorithm to compute a machine schedule, for a given value $M$, is $\mathcal{O}({n_g})$.

\section{Optimization of the Column Generation within the Branch and bound} \label{sec:OPTCGwithBaB}

Since Column Generation is used to provide a lower bound, and solving the master problem does not necessarily lead to obtaining a feasible solution of the bilevel problem, we can stop the column generation process earlier. Thereto, we compute at each iteration $k$ of the column generation a lower bound $LB(k)$ on the optimal value $z_{MP}^*$ of the master problem. Let $r_{\gamma+}$ and $r_{\gamma-}$ denote the outcome values of the pricing problem for a machine with speed
$V_{\gamma}$ that contains a maximum or minimum number of jobs. Then it is well known that
$LB(k) = z_{RMP}(k) + m_{0+}r_{0+} +m_{0-}r_{0-} + m_{1+}r_{1+} +m_{1-}r_{1-} \leq z_{MP}^*$, where $z_{RMP}(k)$ is the optimal value of the restricted master problem.
Thus, we have a global lower bound $LB=\max\limits_k LB(k)$. We stop the column generation at iteration $k$ when $z_{RMP}(k)-LB<1$, in this case the lower bound of a node is given by the weighted number of tardy jobs of the restricted master problem.

Then, when exploring a node in the branch-and-bound algorithm, the scheduling decisions for the jobs have already been made. Consequently, at a given node $\mathcal{N}_{child}$, the dynamic programming procedure begins with the last undecided set of jobs, $G_{g'}$, and the current total completion time, $t'$, for machine $M_i$, reflecting the previously made scheduling decisions. We then use these parameters, $(t', g')$, to compute the minimal reduced cost for machine $M_i$ at node $\mathcal{N}_{child}$. Furthermore, if $\mathcal{B}_1$ is processed on both high and low speed machines, then for the given node $N_{child}$, we might have already scheduled jobs on $\mathcal{B}_1$. Let $\overline{m}_{0+}$ denote the number of high speed machines with a maximum of jobs in the node $N_{child}$. Thus we add the constraint $\sum_{s \in S} b_{s0+} x_s \geqslant \overline{m}_{0+}$, which helps the master problem and gives better experimental results in practice. Remark, at the root node, if the number of jobs to schedule in the first block is greater than the number of low-speed machines, then we can have at least $\max\left(Q-m_1,0\right)=\overline{m}_{0+}$ high-speed machines with a maximum number of jobs.

Moreover, we reuse columns from parent nodes; however, as we explore the search tree using branch-and-bound, a subset of these columns becomes infeasible. Rather than removing each infeasible column immediately, we count the number of times this occurs and define a threshold $\epsilon$. If the proportion of all infeasible columns exceeds $\epsilon$, then we remove them.

Finally, in a search tree, when facing the subproblem of filling the last block, instead of exploring the tree further, we solve it as an assignment problem. Indeed, by the time we need to fill the last block, we already know the completion times of the machines and the jobs that can be scheduled in this last block. Therefore, we can formulate and solve the assignment problem, where our goal is to assign the remaining jobs to the last block in order to minimize the weighted number of tardy jobs.

\section*{Declarations}

\textbf{Conflict of interest} The authors declare that they have no conflict of interest.

\textbf{Ethical approval} This article does not contain any studies with human participants or animals performed by any of the authors.

\bibliographystyle{elsarticle-harv}
\bibliography{./biblio.bib}
\end{document}